\documentclass[12pt]{amsart}

\usepackage{amsmath, amssymb, latexsym}
\usepackage{mathrsfs}
\usepackage{enumerate}
\usepackage[all, knot]{xy}
\xyoption{arc}
\usepackage[usenames,dvipsnames]{color}

\usepackage[colorlinks=true, pdfstartview=FitV, linkcolor=blue,%
citecolor=blue, urlcolor=blue]{hyperref}

\newtheorem{theorem}{Theorem}[section]
\newtheorem{proposition}[theorem]{Proposition}
\newtheorem{corollary}[theorem]{Corollary}

\newtheorem{lemma}[theorem]{Lemma}

\theoremstyle{definition}
\newtheorem{definition}[theorem]{Definition}

\newtheorem{remark}[theorem]{Remark}

\numberwithin{equation}{section}

\def\O{\mathcal O}
\def\Q{\mathcal Q}
\def\P{\mathcal P}

\def\g{\mathfrak g}
\def\h{\mathfrak h}
\def\k{\mathbf k}
\def\Z{\mathbb Z}
\def\A{\mathbb A}

\def\Hom{{\rm Hom}}

\def\wt{{\rm wt}}

\def\Supp{{\rm Supp}}

\def\Ire{I^{\rm re}}
\def\Iim{I^{\rm im}}

\def\Mod{{\rm Mod}}
\def\Proj{{\rm Proj}}
\def\Rep{{\rm Rep}}

\newcommand{\nc}{\newcommand}
\nc{\re}{\mathrm{re}} \nc{\im}{\mathrm{im}}
\nc{\be}{\begin{enumerate}} \nc{\ee}{\end{enumerate}}
\nc{\bnum}{\be[{\rm(i)}]} \nc{\bna}{\be[{\rm(a)}]}
\nc{\eq}{\begin{eqnarray}} \nc{\eneq}{\end{eqnarray}}
\nc{\eqn}{\begin{eqnarray*}} \nc{\eneqn}{\end{eqnarray*}}
\nc{\Oint}{\mathcal{O}_{\mathrm{int}}} \nc{\noi}{\noindent}
\nc{\hs}{\hspace*} \nc{\ba}{\begin{array}} \nc{\ea}{\end{array}}
\nc{\seteq}{\mathbin{:=}} \nc{\set}[2]{\left\{#1\mid #2\right\}}
\nc{\la}{\langle} \nc{\ra}{\rangle} \nc{\E}{\mathcal{E}}
\nc{\F}{\mathcal{F}}
\nc{\bl}{\bigl} \nc{\br}{\bigr}
\nc{\cl}{\colon}

\begin{document}

\title[Categorification of $V(\Lambda)$ over generalized Kac-Moody algebras]
{Categorification of highest weight modules over quantum generalized
Kac-Moody algebras}
\author[Seok-Jin Kang]{Seok-Jin Kang$^{1}$}
\thanks{$^1$ This work was supported by KRF Grant \# 2007-341-C00001 and NRF Grant \# 2010-0010753.}
\address{Department of Mathematical Sciences and Research Institute of Mathematics,
Seoul National University, 599 Gwanak-ro, Gwanak-gu, Seoul 151-747,
Korea} \email{sjkang@snu.ac.kr}

\author[Masaki Kashiwara]{Masaki Kashiwara$^{2}$}
\thanks{$^2$ This work was supported by Grant-in-Aid for
Scientific Research (B) 22340005, Japan Society for the Promotion of
Science.}
\address{Research Institute for Mathematical Sciences, Kyoto University, Kyoto 606-8502,
Japan, and Department of Mathematical Sciences, Seoul National
University, 599 Gwanak-ro, Gwanak-gu, Seoul 151-747, Korea}
\email{masaki@kurims.kyoto-u.ac.jp}

\author[Se-jin Oh]{Se-jin Oh$^{3}$}
\thanks{$^{3}$ This work was supported by BK21 Mathematical Sciences Division and NRF Grant \# 2010-0019516.}
\address{Department of Mathematical Sciences, Seoul National University,
599 Gwanak-ro, Gwanak-gu, Seoul 151-747, Korea}
\email{sj092@snu.ac.kr}

\date{\today}
%\urladdr{}
%\dedicatory{}
\subjclass[2000]{05E10, 16G99, 81R10} \keywords{categorification,
Khovanov-Lauda-Rouquier algebras, cyclotomic quotient, quantum generalized
Kac-Moody algebras}

\begin{abstract}
Let $U_q(\g)$ be a quantum generalized Kac-Moody algebra and let
$V(\Lambda)$ be the integrable highest weight $U_q(\g)$-module with
highest weight $\Lambda$. We prove that the cyclotomic
Khovanov-Lauda-Rouquier algebra $R^\Lambda$ provides a
categorification of $V(\Lambda)$.
\end{abstract}

\maketitle

\vskip 2em

\section*{introduction}
The {\it Khovanov-Lauda-Rouquier algebras}, which were introduced
independently by Khovanov-Lauda and Rouquier, have emerged as a
categorification scheme for quantum groups and their highest weight
modules \cite{KL09, KL11, R08}. That is, if $U_q(\g)$ is the quantum
group associated with a symmetrizable Kac-Moody algebra and $R$ is
the corresponding Khovanov-Lauda-Rouquier algebra, then it was shown
in \cite{KL09, KL11, R08} that there exists an $\A$-algebra
isomorphism
$$U_{\A}^{-}(\g) \simeq [\Proj(R)] =\bigoplus_{\alpha \in \mathtt{Q}^+}
[\Proj(R(\alpha))],$$ where $\A=\Z[q,q^{-1}]$, $U_{\A}^{-}(\g)$ is
the integral form of $U_q^{-}(\g)$, and $[\Proj(R)]$ denotes the
Grothendieck group of the category $\Proj(R)$ of finitely generated
graded projective $R$-modules. Moreover, in
\cite{KL09}, Khovanov and Lauda defined a quotient $R^{\Lambda}$ of
$R$, called the {\em cyclotomic Khovanov-Lauda-Rouquier algebra of
weight $\Lambda$}, and conjectured that there exists a
$U_{\A}(\g)$-module isomorphism
$$V_{\A}(\Lambda) \simeq [\Proj(R^{\Lambda})] =\bigoplus_{\alpha \in \mathtt{Q}^+}
[\Proj(R^{\Lambda}(\alpha))],$$ where $V_{\A}(\Lambda)$ is the
integral form of an irreducible highest weight module $V(\Lambda)$.
It is called the {\it cyclotomic categorification conjecture}.

Brundan and Stroppel (\cite{BS08}) proved a special case of this
conjecture in type $A_n$  and, Brundan and Kleshchev (\cite{BK09})
proved it for type $A_{\infty}$ and $A^{(1)}_n$. In
\cite{LV09}, the crystal version of the conjecture was proved. That
is, Lauda and Vazirani defined the crystal structure on the set of
isomorphism classes of simple objects of the categories $\Rep(R)$
and $\Rep(R^{\Lambda})$ of finite-dimensional $R$-modules and
$R^{\Lambda}$-modules, and showed that they are isomorphic to
$B(\infty)$ and $B(\Lambda)$, respectively. 
Recently, Kang and Kashiwara (\cite{KK11}) proved the
cyclotomic categorification conjecture for {\em all}
symmetrizable Kac-Moody algebras. In \cite{Web10}, Webster gave
a categorification of tensor products of integrable highest weight
modules over quantum groups. 

In \cite{KOP11}, Kang, Oh and Park introduced a family of {\it
Khovanov-Lauda-Rouquier algebras $R$ associated with
Borcherds-Cartan data} and showed that they provide a
categorification of quantum generalized Kac-Moody algebras.
Moreover, for each dominant integral weight $\Lambda$, they defined
the {\it cyclotomic Khovanov-Lauda-Rouquier algebra}
$$R^{\Lambda} =\bigoplus_{\alpha \in \mathtt{Q}^{+}} R^{\Lambda}(\alpha),$$
where $R^{\Lambda}(\alpha) = R(\alpha) \big/ I^{\Lambda}(\alpha)$
and $I^{\Lambda}(\alpha)$ is a two-sided ideal depending on $\Lambda$. 
They proved that the categories of finite-dimensional $R$-modules
and $R^{\Lambda}$-modules have crystal structures that are
isomorphic to $B(\infty)$ and $B(\Lambda)$, respectively.

In this paper,  we prove that
Khovanov-Lauda's cyclotomic categorification conjecture holds for
all generalized Kac-Moody algebras. The main result of this paper
can be summarized as follows. For each $i \in I$, we define two
functors
\begin{align*}
& \E^{\Lambda}_i \cl \Mod(R^{\Lambda}(\beta+\alpha_i)) \longrightarrow \Mod(R^{\Lambda}(\beta)), \\
& \F^{\Lambda}_i \cl \Mod(R^{\Lambda}(\beta)) \longrightarrow \Mod(R^{\Lambda}(\beta+\alpha_i)).
\end{align*}
by
\begin{align*}
& \E^{\Lambda}_i(N) = e(\beta,i)N = e(\beta,i)R^{\Lambda}(\beta+\alpha_i)
                         \otimes_{R^{\Lambda}(\beta+\alpha_i)} N, \\
& \F^{\Lambda}_i(M) = q_i^{1-\langle h_i,\Lambda-\beta\rangle} R^{\Lambda}(\beta+\alpha_i)e(\beta,i)
 \otimes_{R^{\Lambda}(\beta)} M,
\end{align*}
where $q^k$ ($k \in \Z$) denotes the degree shift functor, $M \in
\Mod(R^{\Lambda}(\beta))$ and $N \in
\Mod(R^{\Lambda}(\beta+\alpha_i))$. Then we show that 
the functors $\E^{\Lambda}_i$ and $\F^{\Lambda}_i$ are
well-defined exact functors on $\Proj (R^{\Lambda})$ 
(Theorem~\ref{Thm: Exact}) and they
satisfy the commutation relations (Theorem \ref{Thm: Main})
as operators on $[\Proj(R^{\Lambda})]$
$$[\E^{\Lambda}_i, \F^{\Lambda}_j] = \delta_{i,j}\dfrac{K_i-K^{-1}_i}{q_i-q^{-1}_i}, \quad
\text{ where } \quad K_i|_{[\Proj(R^{\Lambda}(\beta))]} \seteq
q_i^{\langle h_i,\Lambda-\beta \rangle}.$$ Therefore, we obtain a
categorification of the irreducible highest weight module
$V(\Lambda)$ (Theorem \ref{Thm: Categoeification}):
$$[\Proj(R^{\Lambda})] \simeq V_{\A}(\Lambda) \quad \text{ and } \quad
 [\Rep(R^{\Lambda})] \simeq V_{\A}(\Lambda)^{\vee}, $$
where $V_{\A}(\Lambda)^{\vee}$ is the dual of $V_{\A}(\Lambda)$ with respect to a non-degenerate symmetric
 bilinear form on $V(\Lambda)$.

We follow the outline given in \cite{KK11}. The main difference
and difficulty in this paper lie in that we need to deal with a
family of polynomials $\mathcal{P}_{i}$ of degree
$1-\frac{a_{ii}}{2}$ $(i\in I)$ given in \eqref{Def: P_i}, which
makes many of calculations more complicated. Accordingly, the
statements in some lemmas and the one in Theorem \ref{Thm: A nu}
have been modified. The geometric meaning of the polynomials
$\mathcal{P}_{i}$ was recently clarified when the Borcherds-Cartan
datum is symmetric \cite{KKP12}. 

In \cite{KL10}, Khovanov-Lauda gave a precise
description of the relations among the 2-morphisms for 
categorifications of integrable representations of Kac-Moody
algebras, and proved it in the ${\mathfrak{sl}}_{n}$ case.
These relations are proved by Cautis-Lauda \cite{CL12}
for symmetrizable Kac-Moody algebras  under certain conditions. 
It would be an
interesting problem to adapt their relations to the generalized
Kac-Moody algebra case.

This paper is organized as follows. Section \ref{Sec: QGKM and Oint}
contains a brief review of quantum generalized Kac-Moody algebras
and their integrable modules. In Section \ref{Sec: KLR corr QGKM},
we recall the definition of $R$ and its basic properties given in
\cite{KOP11}. In Section \ref{Sec: Fuctors on Mod R}, we define the
functors $E_i$, $F_i$ and $\overline{F}_i$ on $\Mod(R)$ and derive
the relations among them in terms of exact sequences (Theorem
\ref{Thm: Comm E_i F_j}, Theorem \ref{Thm: Comm E_i bar F_j }). In
Section \ref{Sec: cyclo quotient}, we show that the structure of
$R^{\Lambda}$ is compatible with the integrability conditions and
the functors $\mathcal{E}^{\Lambda}_i$ and $\mathcal{F}^{\Lambda}_i$
are well-defined exact functors on $\Proj(R^{\Lambda})$ and
$\Rep(R^{\Lambda})$. In Section \ref{Sec: Categorification}, by
proving the commutation relations among $\mathcal{E}^{\Lambda}_i$
and $\mathcal{F}^{\Lambda}_i$, we conclude that the cyclotomic
Khovanov-Lauda- Rouquier algebra $R^{\Lambda}$ provides a
categorification of the irreducible highest weight module
$V(\Lambda)$ over a quantum generalized Kac-Moody algebra $U_q(\g)$.

\vskip 2em

\section{Quantum generalized Kac-Moody algebras and integrable modules} \label{Sec: QGKM and Oint}

Let $I$ be an index set. A square matrix $\mathtt{A}=(a_{ij})_{i,j \in I}$
with $a_{ij} \in \Z$ is called a {\it Borcherds-Cartan matrix}  if it satisfies
$$
\text{(i) $a_{ii}=2$ or $a_{ii} \in 2\Z_{\le 0}$, \quad (ii) $a_{ij} \le 0$ for $i \neq j$, \quad
 (iii) $a_{ij} =0 $ if and only if $a_{ji}=0$.}$$
An element $i$ of $I$ 
is said to be {\it real} if $a_{ii}=2$ and {\it imaginary}, otherwise. 
We denote by
$\Ire$ the set of all real indices and $\Iim$ the set of all imaginary indices. In this paper, we assume that
$\mathtt{A}$ is {\it symmetrizable};  i.e., there is a diagonal matrix $\mathtt{D}={\rm diag}(d_i \in \Z_{>0} \ |
 \ i \in I)$ such that $\mathtt{DA}$ is symmetric.

A \emph{Borcherds-Cartan datum} $(\mathtt{A},\mathtt{P},\Pi,\Pi^{\vee})$ consists of
\begin{enumerate}
\item[(1)] a Borcherds-Cartan matrix $\mathtt{A}$,
\item[(2)] a free abelian group $\mathtt{P}$, the \emph{weight lattice},
\item[(3)] $\Pi= \{ \alpha_i \in \mathtt{P} \mid \ i \in I \}$, the set of \emph{simple roots},
\item[(4)] $\Pi^{\vee}= \{ h_i \ | \ i \in I  \} \subset \mathtt{P}^{\vee}\seteq\Hom(\mathtt{P},\Z)$, the set of
\emph{simple coroots},
\end{enumerate}
satisfying the following properties:
\begin{enumerate}
\item[(a)] $\langle h_i,\alpha_j \rangle = a_{ij}$ for all $i,j \in I$,
\item[(b)] $\Pi$ is linearly independent,
\item[(c)] for any $i \in I$, there exists $\Lambda_i \in \mathtt{P}$ such that
           $\langle h_j ,\Lambda_i \rangle =\delta_{ij}$ for all $j \in I$.
\end{enumerate}

Let $\h= \mathbb{Q} \otimes_{\Z}\mathtt{P}^{\vee}$. Since $\mathtt{A}$ is symmetrizable,
there is a symmetric bilinear form $( \ | \ )$ on $\h^*$ satisfying
$$(\alpha_i | \alpha_j)= d_i a_{ij} \quad \text{ and } \quad  (\alpha_i | \lambda) = d_i \langle h_i, \lambda
\rangle \quad \text{ for all } i,j \in I, \ \lambda \in \h^*.$$

We denote by $\mathtt{P}^{+}\seteq \{ \lambda \in \mathtt{P} \ | \ \lambda(h_i) \in \Z_{\ge
0}, i \in I \}$ the set of \emph{dominant integral weights}. The
free abelian group $\mathtt{Q}= \oplus_{i \in I} \Z \alpha_i$ is called the
\emph{root lattice}. Set $\mathtt{Q}^+= \sum_{i \in I} \Z_{\ge 0}
\alpha_i$. For $\alpha = \sum k_i \alpha_i \in \mathtt{Q}^+$ and $i \in I$, we define
$$ \Supp(\alpha)= \{ i \in I \ | \ k_i \neq 0 \},
\quad \Supp_i(\alpha) = k_i,
\quad |\alpha|= \sum_{i\in I} k_i.$$

Let $q$ be an indeterminate and $m,n \in \Z_{\ge 0}$. Set $q_i = q^{d_i}$ for $i\in I$. If $i \in
\Ire$, define
\begin{equation*}
 \begin{aligned}
 \ &[n]_i =\frac{ q^n_{i} - q^{-n}_{i} }{ q_{i} - q^{-1}_{i} },
 \ &[n]_i! = \prod^{n}_{k=1} [k]_i ,
\ &\left[\begin{matrix}m \\ n\\ \end{matrix} \right]_i=  \frac{ [m]_i! }{[m-n]_i! [n]_i! }.
 \end{aligned}
\end{equation*}

\begin{definition} \label{Def: GKM}
The {\em quantum generalized Kac-Moody algebra} $U_q(\g)$ associated
with a Borcherds-Cartan datum $(\mathtt{A},\mathtt{P},\Pi,\Pi^{\vee})$ is the associative
algebra over $\mathbb{Q}(q)$ with ${\bf 1}$ generated by $e_i,f_i$ $(i \in I)$ and
$q^{h}$ $(h \in \mathtt{P}^{\vee})$ satisfying following relations:
\bnum
  \item  $q^0=1$, $q^{h} q^{h'}=q^{h+h'} $ for $ h,h' \in \mathtt{P}^{\vee},$
  \item  $q^{h}e_i q^{-h}= q^{\langle h, \alpha_i \rangle} e_i,
          \ q^{h}f_i q^{-h} = q^{-\langle h, \alpha_i \rangle }f_i$ for $h \in \mathtt{P}^{\vee}, i \in I$,
  \item  $e_if_j - f_je_i =  \delta_{ij} \dfrac{K_i -K^{-1}_i}{q_i- q^{-1}_i }, \ \ \mbox{ where } K_i=q_i^{ h_i},$
  \item  $\displaystyle \sum^{1-a_{ij}}_{r=0} \left[\begin{matrix}1-a_{ij} \\ r\\ \end{matrix} \right]_i e^{1-a_{ij}-r}_i
         e_j e^{r}_i =0 \quad \text{ if } i\in \Ire \text{ and } i \ne j, $
  \item $\displaystyle \sum^{1-a_{ij}}_{r=0} \left[\begin{matrix}1-a_{ij} \\ r\\ \end{matrix} \right]_i f^{1-a_{ij}-r}_if_j
        f^{r}_i=0 \quad \text{ if } i \in \Ire \text{ and } i \ne j, $
  \item $ e_ie_j - e_je_i=0,\ f_if_j-f_jf_i =0  \ \ \mbox{ if }a_{ij}=0.$
\end{enumerate}
\end{definition}

Let $U_{q}^+(\g)$ (resp.\ $U_q^{-}(\g)$) be the subalgebra of $U_q(\g)$ generated by the elements $e_i$
(resp.\ $f_i$).

\begin{definition} \label{Def: integrable module}
We define $\O_{int}$ to be the category consisting of $U_q(\g)$-modules
$V$ satisfying the following properties:
\bnum
 \item $V$ has a {\em weight decomposition} with finite-dimensional weight spaces; i.e.,
 $$V = \bigoplus_{\mu \in \mathtt{P}}V_{\mu}\quad \text{with}\quad  \dim V_{\mu}< \infty  ,$$
 where $V_{\mu}=\{v \in V \mid q^h \, v = q^{\langle h,\mu \rangle}v \text{ for all }
 h \in \mathtt{P}^{\vee} \} $ ,
 \item there are finitely many $\lambda_1, \ldots, \lambda_s \in \mathtt{P}$ such that
 $$\wt(V)\seteq\{\mu \in \mathtt{P} \mid V_{\mu} \neq 0 \} \subset \bigcup^{s}_{i=1}(\lambda_i-\mathtt{Q}^+),$$
 \item the action of $f_i$ on $V$ is locally nilpotent for $i\in \Ire$,
 \item if $i \in \Iim$, then $\langle h_i,\mu \rangle \in\Z_{\ge 0}$ for all $\mu \in \wt(V)$,
 \item if $i \in \Iim$ and $\langle h_i,\mu \rangle=0$, then $f_i V_{\mu}=0$,
 \item if $i \in \Iim$ and $\langle h_i,\mu \rangle \le -a_{ii}$, then $e_i V_{\mu}=0.$
\end{enumerate}
\end{definition}

For $\Lambda \in \mathtt{P}$, a $U_q(\g)$-module $V$
is called a {\em highest weight module} with {\em
highest weight $\Lambda$} and {\em highest weight vector $v_{\Lambda}$} if
 there exists  $v_{\Lambda} \in V$ such that
\begin{align*}
(1) \   V=U_q(\g) v_{\Lambda}, \ \ \ (2) \   q^h v_{\Lambda} =
q^{\langle h, \Lambda \rangle} v_{\Lambda} \ \ \text{for all} \ h
\in \mathtt{P}^{\vee}, \ \ \ (3) \ e_i v_{\Lambda} =0 \ \ \text{ for
all } i \in I.
\end{align*}
For $\Lambda \in \mathtt{P}^+$, let us denote by
$V(\Lambda)$ the $U_q(\g)$-module  generated by $v_{\Lambda}$ with the defining relation:
\eqn
\parbox{60ex}{
\bna
\item $v_\Lambda$ is a highest weight vector of weight $\Lambda$,
\item $f_i^{\langle h_i,\Lambda \rangle +1 } v_{\Lambda}=0$ for any $i \in \Ire$,
 \item  $f_iv_{\Lambda}=0$  if $\langle h_i, \Lambda \rangle =0$.
\ee
}
\eneqn

\medskip
\begin{proposition} [{\cite{BKM98,JKK05,Kang95}}] \label{Prop: HW module} \
\bnum
\item For any $\Lambda \in \mathtt{P}^+$, $V(\Lambda)$ is an irreducible $U_q(\g)$-module.
\item If $V$ is a highest weight module in $\Oint$, then $V$ is isomorphic to $V(\Lambda)$ for some
$\Lambda \in \mathtt{P}^+$.
\item Any module in $\Oint$ is semisimple.
\end{enumerate}
\end{proposition}

Let $\phi$ be the anti-automorphism of $U_q(\g)$ given by
$$ \phi(e_i) = f_i, \ \ \phi(f_i) = e_i  \ \text{ and } \ \phi(q^h)= q^h. $$
In \cite{KOP11a}, it was shown that there exists a unique
non-degenerate symmetric bilinear form $( \ , \ )$ on $V(\Lambda)$
$(\Lambda \in \mathtt{P}^+)$ 
satisfying
\begin{align} \label{Eq: Pairing}
  (v_{\Lambda},v_{\Lambda})=1, \ \
  (x u,v)=(u,\phi(x)v)\ \text{for $x \in U_q(\g)$ and $u,v\in V(\Lambda)$.}
 \end{align}
Set $\A=\Z[q,q^{-1}]$. We define the $\A$-form $V_\A(\Lambda)$ of $V(\Lambda)$ to be
$$ V_\A(\Lambda) = U_\A(\g) v_\Lambda,$$
where $U_\A(\g)$ is the $\A$-subalgebra of $U_q(\g)$ defined in
\cite[Section 9]{JKK05}.

The dual of $V_\A(\Lambda)$ is defined to be
$$ V_\A(\Lambda)^{\vee} = \{ v \in V(\Lambda) \ | \ (u,v) \in \A \text{ for all } u \in V_\A(\Lambda) \}.$$

\vskip 2em

\section{The Khovanov-Lauda-Rouquier algebras for generalized Kac-Moody algebras}
\label{Sec: KLR corr QGKM}

We take a graded commutative ring $\k= \oplus_{n \in \Z_{\ge 0}} \k_{n}$
as a base ring.
For a given Borcherds-Cartan datum $(\mathtt{A},\mathtt{P},\Pi,\Pi^{\vee})$,
we take $\Q_{i,j}(u,v) (i,j \in I)$ in $\k[u,v]$ 
such that $\Q_{i,j}(u,v)=\Q_{j,i}(v,u)$ and
$\Q_{i,j}(u,v)$ has the form
\begin{align*}
\Q_{i,j}(u,v)=
\begin{cases}
\quad  0 & \text{ if } i =j, \\
%\quad  t_{ij} & \text{ if } i \neq j, \ a_{ij}=0, \\
%t_{ij} u^{-a_{ij}} + {\displaystyle \sum_{ \substack{ p,q > 0,\\ d_ip + d_jq = -(\alpha_i | \alpha_j)}}} %w_{ij}^{pq} u^p v^q + t_{ji} v^{-a_{ji}} & \text{ otherwise,}
\displaystyle \sum_{d_ip + d_jq \le -(\alpha_i | \alpha_j)} t_{i,j}^{p,q} u^p v^q
& \text{ if } i \neq j,
\end{cases}
\end{align*}
where $ t^{-a_{ij},0}_{i,j} \in \k_0^{\times} $
and $t_{i,j}^{p,q} \in \k_{-2((\alpha_i|\alpha_j)+d_ip+d_jq)}$ with $t_{i,j}^{p,q} = t_{j,i}^{q,p}$.

For all $i \in I$, we take polynomials $\P_i(u,v)$ in
$\k[u,v]$ which have the form
\begin{equation} \label{Def: P_i}
\begin{aligned}
\P_i(u,v)=
%\begin{cases}
%\quad  1 & \text{ if } i \in \Ire , \\
\sum_{p+q \le 1-\frac{a_{ii}}{2}} w_i^{p,q} u^p v^q,  %& \text{ if } i \in \Iim,
%\end{cases}
\end{aligned}
\end{equation}
where $w^{1-\frac{a_{ii}}{2},0}_{i},w^{0,1-\frac{a_{ii}}{2}}_{i} \in
\k_0^{\times}$ and $w_i^{p,q}\in \k_{2d_i(1-p-q-\frac{a_{ii}}{2})}$.

\begin{remark} \label{Rem: condition of P}
In \cite{KOP11}, it was assumed that $\P_i(u,v)$ is a symmetric
homogeneous polynomial. But, in this paper, we do not assume that
$\P_i(u,v)$ is symmetric. Instead, we put more restrictions on the
leading terms of $\P_i(u,v)$.  Accordingly, the defining relations
of Khovanov-Lauda-Rouquier algebras in Definition~\ref{def:KLR}
below are modified from the ones in \cite{KOP11}.  This choice
will be used in a critical way in the proof of Lemma \ref{Lem:
nilpotency} and Lemma \ref{Lem: Rel between F and S}. The main results of \cite{KOP11} are
still valid after this modification.
\end{remark}

We denote by $S_n= \langle s_1, \ldots, s_{n-1} \rangle$ the
symmetric group on $n$ letters, where $s_i=(i,i+1)$ is the
transposition. Then $S_n$ acts on $I^n$ and $\k[x_1,\ldots,x_n]$ in
a natural way.

We define the operator $\partial_a$ on $\k[x_1,\ldots,x_n]$, by
$$ \partial_{a,b}f = \dfrac{s_{a,b}f-f}{x_a-x_b}, \quad \partial_{a}:=\partial_{a,a+1},$$
where $s_{a,b}=(a,b)$ is the transposition.

For the sake of simplicity, we assume that $I$ is a finite set.
\begin{definition}[{\cite{KOP11}}] \label{def:KLR} \
%\begin{itemize}
%\item[(a)]
The Khovanov-Lauda-Rouquier algebra $R(n)$ of degree $n$ associated with
the data $(\mathtt{A},\mathtt{P},\Pi,\Pi^{\vee})$, $(\Q_{i,j})_{i,j
\in
 I}$ and $(\P_i)_{i \in I}$ is the associative algebra over $\k$ generated by $e(\nu)$ $(\nu \in I^n)$, $x_k$ $(
 1 \le k \le n)$, $\tau_{\ell}$ $(1 \le \ell \le n-1)$ with following relations:
\begin{equation} \label{Eq: Relation1}
 \begin{aligned}
& e(\nu)e(\nu')=\delta_{\nu,\nu'}e(\nu), \quad \sum_{\nu \in I^n}e(\nu)=1, \\
& x_k x_l=x_l x_k, \quad x_k e(\nu) =e(\nu) x_k, \\
& \tau_\ell e(\nu)=e(s_\ell \nu) \tau_\ell, \quad \tau_k \tau_\ell = \tau_\ell \tau_k \text{ if } |k-\ell|>1, \\
& \tau^2_k e(\nu) = \begin{cases}
(\partial_k \P_{\nu_k}(x_k,x_{k+1})) \tau_k e(\nu) & \text{ if } \nu_k = \nu_{k+1}, \\
\Q_{\nu_k,\nu_{k+1}}(x_k,x_{k+1})e(\nu), & \text{ if } \nu_k \neq \nu_{k+1}.
\end{cases}
\end{aligned}
\end{equation}
\begin{equation} \label{Eq: Relation2}
 \begin{aligned}
& (\tau_k x_\ell - x_{s_k(\ell)}\tau_k)e(\nu)=
\begin{cases}
-\P_{\nu_k}(x_k,x_{k+1})e(\nu) & \text{ if } \ell=k, \ \nu_k=\nu_{k+1}, \\
\P_{\nu_k}(x_k,x_{k+1})e(\nu) & \text{ if } \ell=k+1, \ \nu_k=\nu_{k+1}, \\
\qquad \qquad 0 & \text{ otherwise.}
\end{cases}
\end{aligned}
\end{equation}

\begin{equation} \label{Eq: Braid Relation}
 \begin{aligned}
&(\tau_{k+1}\tau_{k}\tau_{k+1}-\tau_{k}\tau_{k+1}\tau_{k})e(\nu)  \\
& =
\begin{cases}
\P_{\nu_k}(x_k,x_{k+2})\overline{\Q}_{\nu_k,\nu_{k+1}}(x_k,x_{k+1},x_{k+2})
e(\nu)
& \text{ if } \nu_{k}=\nu_{k+2}\neq \nu_{k+1},\\
\overline{\P}'_{\nu_k}(x_k,x_{k+1},x_{k+2})\tau_k e(\nu) +
\overline{\P}''_{\nu_k}(x_k,x_{k+1},x_{k+2})\tau_{k+1}e(\nu)
& \text{ if } \nu_{k}=\nu_{k+1}= \nu_{k+2},\\
\qquad \qquad  0 & \text{ otherwise,}
\end{cases}
\end{aligned}
\end{equation}
where
\begin{align*}
\overline{\P}'_i(u,v,w)=\overline{\P}'_i(v,u,w) &\seteq
\frac{\mathcal{P}_{i}(v,u)\mathcal{P}_{i}(u,w)}{(u-v)(u-w)} +
\frac{\mathcal{P}_{i}(u,w)\mathcal{P}_{i}(v,w)}{(u-w)(v-w)} -
\frac{\mathcal{P}_{i}(u,v)\mathcal{P}_{i}(v,w)}{(u-v)(v-w)}, \\
\overline{\P}''_i(u,v,w)=\overline{\P}''_i(u,w,v) &\seteq
-\frac{\mathcal{P}_{i}(u,v)\mathcal{P}_{i}(u,w)}{(u-v)(u-w)}
-\frac{\mathcal{P}_{i}(u,w)\mathcal{P}_{i}(w,v)}{(u-w)(v-w)}+
\frac{\mathcal{P}_{i}(u,v)\mathcal{P}_{i}(v,w)}{(u-v)(v-w)}, \\
\overline{\mathcal{Q}}_{i,j}(u,v,w)&\seteq
\dfrac{\mathcal{Q}_{i,j}(u,v) - \mathcal{Q}_{i,j}(w,v)}{u-w}.
\end{align*}
\end{definition}

\noindent The $\Z$-grading on $R(n)$ is given by
$$ \deg(e(\nu)) =0 , \quad \deg(x_ke(\nu))=2 d_{\nu_k}, \quad \deg(\tau_\ell e(\nu)) =
-(\alpha_{\nu_\ell}|\alpha_{\nu_{\ell+1}})$$ for all $\nu \in I^{n}$
, $1 \le k \le n$ and $1 \le \ell < n$.

For $\nu=(\nu_1,\ldots,\nu_n) \in I^{n}$ and $1 \le m \le n$, we define
{\allowdisplaybreaks
\begin{align*}
&\nu_{<m} = (\nu_1,\ldots,\nu_{m-1}) , \qquad \nu_{\le m} = (\nu_1,\ldots,\nu_{m}), \\
&\nu_{>m} = (\nu_{m+1},\ldots,\nu_{n}) , \qquad \nu_{\ge m} = (\nu_m,\ldots,\nu_{n}).
\end{align*}

For pairwise distinct $a,b,c \in \{ 1, \ldots n \}$, let us define
\begin{align*}
& %\Q_{a,b}= \sum_{\nu \in I^n} \Q_{\nu_a,\nu_b}(x_a,x_b)e(\nu), \quad
 e_{a,b}= \sum_{ \substack{\nu \in I^n, \\ \nu_a=\nu_b } }e(\nu), \quad
\P_{a,b}= \sum_{ \substack{\nu \in I^n, \\ \nu_a=\nu_b } } \P_{\nu_a}(x_a,x_b)e(\nu), \\
& \overline{\Q}_{a,b,c}= \sum_{ \substack{\nu \in I^n, \\ \nu_a=\nu_c \neq \nu_b} }
    \dfrac{\Q_{\nu_a,\nu_b}(x_a,x_b)-\Q_{\nu_a,\nu_b}(x_c,x_b)}{x_a -x_c} e(\nu),
    \quad \overline{\Q}_{a}\seteq\overline{\Q}_{a,a+1,a+2},\\
& \overline{\P}'_{a,b,c}= \sum_{ \substack{\nu \in I^n, \\ \nu_a=\nu_{b}=\nu_{c}} }
      \overline{\P}'_{\nu_a}(x_a,x_b,x_c)e(\nu), \quad \overline{\P}'_{a}\seteq\overline{\P}'_{a,a+1,a+2},\\
& \overline{\P}''_{a,b,c}= \sum_{ \substack{\nu \in I^n, \\ \nu_a=\nu_{b}=\nu_{c}} }
      \overline{\P}''_{\nu_a}(x_a,x_b,x_c)e(\nu), \quad \overline{\P}''_{a}\seteq\overline{\P}''_{a,a+1,a+2}.
\end{align*}
}
Then we have
$$ %\Q_{a,b}=\Q_{b,a}, \quad \tau^2_a=\Q_{a,a+1}, \quad
\tau_{a+1}\tau_{a}\tau_{a+1}-\tau_{a}\tau_{a+1}\tau_{a}=
\overline{\Q}_a \P_{a,a+2}+ \overline{\P}'_a \tau_a +
\overline{\P}''_a \tau_{a+1}.$$
Note that we have $\overline{\P}'_a
\tau_a=\tau_a\overline{\P}'_a$ and $\overline{\P}''_a
\tau_{a+1}=\tau_{a+1}\overline{\P}''_a $ by the formula \eqref{Eq:
partial} below.

We define the operator, also denoted by $\partial_{a,b}$, on $\oplus_{\nu \in I^n} \k[x_1,\ldots,x_n]e(\nu)$, by
$$ \partial_{a,b}f = \dfrac{s_{a,b}f-f}{x_a-x_b}e_{a,b}, \quad \partial_a\seteq\partial_{a,a+1}.$$
%where $s_{a,b}=(a,b)$ is the transposition.
Then we obtain
\begin{align} \label{Eq: partial}
\tau_a f - (s_af)\tau_a =f \tau_a -\tau_a(s_a f) = (\partial_a f) \P_{a,a+1}.
\end{align}

For $\beta \in \mathtt{Q}^+$ with $|\beta|=n$, we set
$$ I^\beta=\{ \nu=(\nu_1,\ldots,\nu_n) \in I^n \ | \ \alpha_{\nu_1}+ \cdots + \alpha_{\nu_n}=\beta \}.$$
We define
\begin{align*}
& R(m,n)=R(m)\otimes_{\k}R(n)\subset R(m+n), \\
& e(n)=\sum_{\nu \in I^n}e(\nu), \quad e(\beta)=\sum_{\nu \in I^\beta}e(\nu), \quad R(\beta)=e(\beta)R(n), \\
& e(n,i)=\sum_{ \substack{\nu \in I^{n+1}, \\ \nu_{n+1}=i} }e(\nu), \quad
  e(i,n)=\sum_{ \substack{\nu \in I^{n+1}, \\ \nu_{1}=i} }e(\nu),\\
& e(\beta,i) = e(\beta+\alpha_i)e(n,i), \quad  e(i,\beta) = e(\beta+\alpha_i)e(i,n).
\end{align*}
Then $R= \bigoplus_{\alpha \in \mathtt{Q}^+} R(\alpha)$.

\smallskip
\begin{proposition}[{\cite{KOP11}}] \label{Prop: essence of KOP1} \
\bnum
\item $R(\alpha)$ is noetherian.
\item There are only finitely many irreducible graded $R(\alpha)$-modules up to isomorphism and grading shift.
Moreover, all the irreducible graded $R(\alpha)$-modules are
finite-dimensional.
\item The Krull-Schmidt unique direct sum decomposition property holds for all finitely generated graded $R(\alpha)$-modules.
\end{enumerate}
\end{proposition}

In the rest of this section, assume that $\k_0$ is a field.
Let $\Mod(R(\alpha))$ (resp.\ $\Proj(R(\alpha))$, $\Rep(R(\alpha))$) be the category of arbitrary
(resp.\ finitely generated projective, finite-dimensional over $\k_0$)
graded left $R(\alpha)$-modules. The morphisms in these
categories are degree preserving homomorphisms. Define
$$ [\Proj(R)]\seteq\bigoplus_{\alpha \in \mathtt{Q}^+} [\Proj(R(\alpha))] \text{ and }
[\Rep(R)]\seteq\bigoplus_{\alpha \in \mathtt{Q}^+}
[\Rep(R(\alpha))], $$ where $[\Proj(R(\alpha))]$ (resp.
$[\Rep(R(\alpha))]$) is the Grothendieck group of $\Proj(R(\alpha))$
(resp. $\Rep(R(\alpha))$). We can define the degree shift functors
$q^m$ ($m \in \Z$) on $\Mod(R(\alpha))$ given as follows: For $M =
\oplus_{k \in \Z} M_k$,
$$ q^m(M)\seteq M\langle -m\rangle \quad\text{where 
$M\langle m \rangle_k=M_{k+m}$.}$$
Then one can define $\A$-module structures on $[\Proj(R)]$ and
$[\Rep(R)]$. The following theorem provides a
categorification of quantum generalized Kac-Moody algebras.

\begin{theorem}[{\cite{KOP11}}] \label{Thm: main result of KOP11}
There is an injective $\A$-algebra homomorphism
$$ U^-_\A(\g) \hookrightarrow [\Proj(R)].$$
It is an isomorphism if $a_{ii} \neq 0$ for any $i \in I$.
\end{theorem}

\vskip 2em

\section{The Functors $E_i$ and $F_i$ on $\Mod(R)$.} \label{Sec: Fuctors on Mod R}

%We assume that the base ring $\k$ is a commutative ring.

{}From the natural embedding $R(\beta) \otimes_{\k} R(\alpha_i) \hookrightarrow R(\beta+\alpha_i)$, we obtain the
functors
\begin{align*}
& E_i\cl \Mod(R(\beta+\alpha_i)) \to \Mod(R(\beta)),\\
& F_i\cl \Mod(R(\beta)) \to \Mod(R(\beta+\alpha_i))
\end{align*}
given by
\begin{align*}
 E_i(M)=M &\mapsto e(\beta,i)M \simeq e(\beta,i)R(\beta+\alpha_i)\otimes_{R(\beta+\alpha_i)}M, \\
 F_i(N)=N &\mapsto R(\beta+\alpha_i)e(\beta,i)\otimes_{R(\beta)}N
\end{align*}
for $M \in \Mod(R(\beta+\alpha_i))$ and $N \in \Mod(R(\beta))$.

Let $\xi_n\cl R(n) \to R(n+1)$ be the algebra monomorphism given by
$$ \xi_n(x_k) = x_{k+1}, \quad \xi_n(\tau_\ell)=\tau_{\ell+1}, \quad
   \xi_n(e(\nu))= \sum_{i \in I} e(i,\nu) $$
for all $1 \le k \le n$, $1 \le \ell <n$ and $\nu \in I^n$. Let $R^1(n)$ be the image of $\xi_n$.
Then for each $i \in I$, we can define the functor
$$ \overline{F}_i\cl \Mod(R(\beta)) \to \Mod(R(\beta)+\alpha_i) \text{ by } N \mapsto
R(\beta+\alpha_i)e(i,\beta)\otimes_{R(\beta)}N.$$
Here, the right $R(\beta)$-module structure on $R(\beta+\alpha_i)e(i,\beta)$ is given by the embedding
$$ R(\beta) \overset{\sim}{\to} R^1(\beta) \hookrightarrow R(\beta+\alpha_i).$$

From now on, we will investigate the relationship among these
functors.

\begin{proposition}[{\cite[Corollary 2.5]{KOP11}}] We have a decomposition
$$R(n+1)= \bigoplus_{a=1}^{n+1}R(n,1)\tau_n \cdots \tau_a = \bigoplus_{a=1}^{n+1}R(n)\otimes\k[x_{n+1}] \tau_n \cdots \tau_a.$$
Furthermore, $R(n+1)$ is a free $R(n,1)$-module of rank $n+1$.
\end{proposition}

\begin{lemma} \label{Lem: Mod R(n,1)}
For $1 \le a \le n$, $f(x_n) \in \k[x_n]$ and $y \in R(n)$, we have
$$ \tau_a \cdots \tau_{n-1}f(x_n)\tau_n y \equiv \tau_a \cdots \tau_{n-1}\tau_n f(x_{n+1})y \
\quad {\rm mod} \ R(n,1).$$
\end{lemma}

\begin{proof}
By $\eqref{Eq: partial}$, we have
\begin{align} \label{Eq: Mod R(n,1)}
\tau_a \cdots \tau_{n-1}f(x_n)\tau_n y =\tau_a \cdots \tau_{n-1}(\tau_n f(x_{n+1}) +(\partial_n f)
\P_{n,n+1})y.
\end{align}
Since $(\partial_n f) \P_{n,n+1} \in \sum_{\nu \in I^{n+1}}
\k[x_n,x_{n+1}]e(\nu) \subset R(n,1)$, the second term in the
right-hand side of $\eqref{Eq: Mod R(n,1)}$ is equal to
$$ \tau_a \cdots \tau_{n-1}(\partial_n f) \P_{n,n+1} y \equiv 0 \quad {\rm mod} \ R(n,1).$$
Hence our assertion holds.
\end{proof}

\begin{proposition} \label{Prop: twist by tau n}
The homomorphism $ R(n) \otimes_{R(n-1)} R(n) \longrightarrow R(n+1)$
given by
$$ x \otimes y \longmapsto x \tau_n y \quad (x,y \in R(n))$$
induces an isomorphism of $(R(n),R(n))$-bimodules
\begin{align}\label{Eq: E_iF_i}
 R(n) \otimes_{R(n-1)}R(n) \oplus R(n,1) \overset{\sim}{\to} R(n+1).
\end{align}
\end{proposition}

\begin{proof}
Using Lemma \ref{Lem: Mod R(n,1)}, we can apply a similar argument
given in \cite[Proposition 3.3]{KK11}
\end{proof}

\begin{corollary}\label{Cor: commutiation E_iF_j} There exists a natural isomorphism
\begin{align*}
& e(n,i)R(n+1)e(n,j) \\
&= \begin{cases}
  q_i^{-a_{ij}} R(n) e(n-1,j) \otimes_{R(n-1)} e(n-1,i)R(n) & \text{ if } i \neq j,\\
  q_i^{-a_{ii}} R(n) e(n-1,i) \otimes_{R(n-1)} e(n-1,i)R(n) \oplus e(n,i)R(n,i)e(n,i) & \text{ if } i = j, \end{cases}
\end{align*}
where $q$ is the degree shift functor and $q_i=q^{d_i}$.
\end{corollary}

\begin{proof} By applying the exact functor $e(n,i) \ \bullet \ e(n,j)$ on $\eqref{Eq: E_iF_i}$, we obtain
\begin{align*}
& \ e(n,i)R(n+1)e(n,j) \cong  e(n,i)(R(n) \otimes_{R(n-1)}R(n) \oplus R(n,1))e(n,j) \\
& \cong e(n,i)R(n)\otimes_{R(n-1)}R(n)e(n,j) \oplus \delta_{ij} e(n,i) R(n,1) e(n,j) \\
& \cong R(n)e(n,i)\otimes_{R(n-1)}e(n,j)R(n) \oplus \delta_{ij} e(n,i) R(n,1) e(n,j) \\
& \cong R(n)e(n-1,j)\otimes_{R(n-1)}e(n-1,i)R(n) \oplus \delta_{ij} e(n,i) R(n,1) e(n,j).
\end{align*}
The grading-shift $q_i^{-a_{ij}}=q^{-(\alpha_i|\alpha_j)}$ arises
from $e(n,i)\tau_n e(n,j)$.
\end{proof}

Note that the kernels of $E_iF_j$ and $F_jE_i$ are given by
\begin{equation}\label{Eq: Kernels E_iF_j and F_jE_i}
\begin{aligned}
& e(n,i)R(n+1)e(n,j)e(\beta)=e(\beta,i)R(\beta+\alpha_j)e(\beta,j), \\
& R(n)e(n-1,j)\otimes_{R(n-1)}e(n-1,i)R(n)e(\beta)=R(\beta-\alpha_i+\alpha_j) e(\beta-\alpha_i,j),
\end{aligned}
\end{equation}
respectively. The following theorem is an immediate consequence of
Corollary \ref{Cor: commutiation E_iF_j}.

\begin{theorem} \label{Thm: Comm E_i F_j}
There exist natural isomorphisms
\begin{align*}
E_iF_j \overset{\sim}{\to}
\begin{cases}
q_i^{-a_{ij}}F_jE_i & \text{ if } i \neq j, \\
q_i^{-a_{ii}}F_iE_i \oplus {\rm Id}\otimes \k[t_i] & \text{ if } i = j,
\end{cases}
\end{align*}
where $t_i$ is an indeterminate of degree $2 d_i$ and
${\rm Id}\otimes \k[t_i]\cl\Mod(R(\beta)) \to \Mod(R(\beta))$ is the
functor $M \mapsto M \otimes \k[t_i]$.
\end{theorem}

\begin{proposition} \label{Pro: R(n)R1(n)}
There exists an injective homomorphism
$$ \Phi\cl R(n) \otimes_{R^1(n-1)} R^1(n) \to R(n+1) \quad \text{ given by } x \otimes y \mapsto xy.$$
Moreover, its image $R(n)R^1(n)$ has decomposition
$$R(n)R^1(n) = \bigoplus_{a=2}^{n+1}R(n,1)\tau_n \cdots \tau_2 =
                  \bigoplus_{a=0}^{n-1}\tau_a \cdots \tau_1 R(1,n). $$
\end{proposition}

\begin{proof} The proof is the same as that of \cite[Proposition 3.7]{KK11}
\end{proof}

By Proposition \ref{Pro: R(n)R1(n)}, there exists a map $\varphi_1\cl R(n+1) \to R(n) \otimes
\k[x_{n+1}]$ given by \begin{equation} \label{Eq: varphi 1}
\begin{aligned}
R(n+1) \to & {\rm Coker}(\Phi)
\cong \dfrac{\bigoplus_{a=1}^{n+1}R(n,1)\tau_n\cdots \tau_a}{\bigoplus_{a=2}^{n+1}R(n,1)\tau_n\cdots \tau_a}
\overset{\sim}{\gets} R(n,1)\tau_n\cdots \tau_1 \overset{\sim}{\gets} R(n,1) \\
& \cong R(n) \otimes \k[x_{n+1}] \cong R(n) \otimes \k[t_i].
\end{aligned}
\end{equation}

Similarly, there is an another map $\varphi_2\cl R(n+1) \to \k[x_1]\otimes R(n)$ given by
\begin{equation} \label{Eq: varphi 2}
\begin{aligned}
R(n+1) \to & {\rm Coker}(\Phi)
\cong \dfrac{\bigoplus_{a=0}^{n}\tau_a\cdots \tau_1 R(1,n)}{\bigoplus_{a=0}^{n-1}\tau_a\cdots \tau_1R(1,n)}
\overset{\sim}{\gets} \tau_n\cdots \tau_1 R(1,n) \overset{\sim}{\gets} R(1,n) \\
& \cong \k[x_{1}] \otimes R(n) \cong \k[t_i] \otimes R(n).
\end{aligned}
\end{equation}

%%%%%%%(2)%%%%%%%

We claim that the maps $\varphi_1$ and $\varphi_2$ coincide with each other, which is an immediate consequence
of the following lemma. When $a_{ii}=2$ for all $i \in I$,
the proof easily follows from $\eqref{Eq: Relation2}$ and
$\eqref{Eq: Braid Relation}$. However, when $a_{ii} \neq 2$ for some $i \in I$, the verification becomes
more complicated.

\begin{lemma} \label{Lem: conincidence} For all $1 \le k \le n$ and $1 \le \ell \le n-1$,
\begin{itemize}
\item[(a)] $x_k \tau_n \cdots \tau_1 \equiv \tau_n \cdots \tau_1 x_{k+1}$,
\item[(b)] $\tau_\ell \tau_n \cdots \tau_1 \equiv \tau_n \cdots \tau_1 \tau_{\ell+1}$,
\item[(c)] $x_{n+1} \tau_n \cdots \tau_1 \equiv \tau_n \cdots \tau_1 x_1$ ${\rm mod} \ R(n)R^1(n).$
\end{itemize}
\end{lemma}

\begin{proof}
We will verify that
\begin{equation} \label{Eq: claim in Mod R(n)R1(n)}
\begin{aligned}
& \parbox{70ex}{for $f \in \k[x_{1},\cdots,x_{n+1}]$,\\
$\tau_n\tau_{n-1} \cdots \tau_k \  f \ \tau_\ell \cdots \tau_1 \equiv 0 \mod R(n)R^1(n)$
\text{ if } $\ell+2\le k\le n+1$.}
\end{aligned}
\end{equation}
We shall prove this by using downward induction on $k$. If $k=n+1$, it is trivial.

Assume that $k\le n$ and our assertion is true for $k+1$. Then we have
\begin{equation} \label{Eq: Ind k+1}
\begin{aligned}
\tau_n \cdots \tau_k f \tau_\ell \cdots \tau_1
&= \tau_n \cdots \tau_{k+1}(s_k(f)\tau_k+f')\tau_\ell \cdots \tau_1 \\
&= \tau_n \cdots \tau_{k+1}s_k(f)\tau_\ell \cdots \tau_1 \tau_k
   + \tau_n \cdots \tau_{k+1}f'\tau_\ell \cdots \tau_1
\end{aligned}
\end{equation}
for some $f' \in \k[x_{1},\cdots,x_{n+1}]$. Since $\tau_k \in R^1(n)$, all the terms in the right-hand
side of $\eqref{Eq: Ind k+1}$ are 0 ${\rm mod} \ R(n)R^1(n)$ by the induction hypothesis.
Hence our assertion holds.

\medskip
\noi
(a) For $1 \le k \le n$, we have
\begin{align*}
x_k \tau_n \cdots \tau_1 & =  \tau_n \cdots \tau_{k+1} x_k \tau_k \cdots \tau_1 \\
&= \tau_n \cdots \tau_{k+1} \tau_k x_{k+1} \tau_{k-1} \cdots \tau_1
   -\tau_n \cdots \tau_{k+1} \P_{k,k+1}\tau_{k-1} \cdots \tau_1 .
\end{align*}
Then the second term is $0$
$ {\rm mod} \ R(n)R^1(n)$ by $\eqref{Eq: claim in Mod R(n)R1(n)}$, and the first term
is equal to
$$(\tau_n \cdots \tau_{k+1} \tau_k )(\tau_{k-1}\cdots \tau_1)x_{k+1},$$
which implies our first assertion.

\medskip
\noi
(b) For $1 \le \ell \le n-1$, we have
\begin{align*}
&\tau_\ell \tau_n \cdots \tau_1 =  \tau_n \cdots \tau_{\ell+2}\tau_{\ell}\tau_{\ell+1}\tau_{\ell} \cdots \tau_1 \\
& = \tau_n \cdots \tau_{\ell+2}
    (\tau_{\ell+1}\tau_{\ell}\tau_{\ell+1}-\overline{\Q}_{\ell}\P_{\ell,\ell+2}
    -\overline{\P}'_{\ell}\tau_{\ell}-\tau_{\ell+1}\overline{\P}''_{\ell})
    \tau_{\ell-1} \cdots \tau_1 \\
& = \tau_n \cdots \tau_1 \tau_{\ell+1}
-\tau_n \cdots \tau_{\ell+2}(\overline{\Q}_{\ell}\P_{\ell,\ell+2})
    \tau_{\ell-1} \cdots \tau_1 \\
& \ \  -\tau_n \cdots \tau_{\ell+2}
(\overline{\P}'_{\ell})\tau_{\ell} \cdots \tau_1  -\tau_n \cdots
\tau_{\ell+1} (\overline{\P}''_{\ell})\tau_{\ell-1} \cdots \tau_1.
\end{align*}
By $\eqref{Eq: claim in Mod R(n)R1(n)}$, the terms except the first one are 0 $ {\rm mod} \ R(n)R^1(n)$.

(c) If $k=n+1$, we have
\begin{align*}
x_{n+1}\tau_n \cdots \tau_1 & = (\tau_n x_n + \P_{n,n+1}) \tau_{n-1} \cdots \tau_1 \\
& = \tau_n x_n \tau_{n-1} \cdots \tau_1+ \P_{n,n+1}\tau_{n-1} \cdots \tau_1 \\
& \equiv \tau_n x_n \tau_{n-1} \cdots \tau_1 \\
& \quad \quad  \vdots  \\
& \equiv \tau_n \cdots \tau_1 x_1 \quad {\rm mod} \ R(n)R^1(n).
\end{align*}
\end{proof}

As an immediate corollary, we obtain

\begin{corollary} \label{Lem: varphi 2}
There is an exact sequence of $(R(n),R(n))$-bimodules
\begin{align} \label{Eq: E_i Bar F_i}
 0 \to R(n) \otimes_{R(n-1)} R(n) \to R(n+1) \overset{\varphi}{\to} R(n) \otimes \k[t_i] \to 0,
\end{align}
where the map $\varphi$ is given by $\eqref{Eq: varphi 1}$ or
$\eqref{Eq: varphi 2}$. Here, the right $R(n)$-module structure on
$R(n+1)$ is given by the embedding $\xi_n\cl R(n) \overset{\sim}{\to} R^1(n)
\hookrightarrow R(n+1)$. Moreover, both the left multiplication by
$x_{n+1}$ and the right multiplication by $x_1$ on $R(n+1)$ are
compatible with the multiplication by $t_i$ on $R(n) \otimes
\k[t_i]$.
\end{corollary}

By applying the exact functor $e(\beta+\alpha_j-\alpha_i,i) \
\bullet \ e(j,\beta)$ on $\eqref{Eq: E_i Bar F_i}$, Corollary
\ref{Lem: varphi 2} yields the following theorem.

\begin{theorem} \label{Thm: Comm E_i bar F_j } \
\bnum
\item There is a natural isomorphism
 $$ \overline{F}_j E_i \overset{\sim}{\to} E_i \overline{F}_j \quad \text{ for } i \neq j.$$
\item There is an exact sequence in $\Mod(R(\beta))$:
 $$ 0\to \overline{F}_i E_i M \to E_i \overline{F}_i M \to q^{-(\alpha_i | \beta )}M \otimes \k[t_i] \to 0,$$
 which is functorial in $M$.
\end{enumerate}
\end{theorem}

\vskip 2em

\section{ The cyclotomic quotient $R^{\Lambda}$ } \label{Sec: cyclo quotient}

In this section, we define the cyclotomic Khovanov-Lauda-Rouquier
algebra $R^{\Lambda}$ and the functors $E^{\Lambda}_i$,
$F^{\Lambda}_i$ on $\Mod(R^{\Lambda})$. We investigate the structure
of $R^{\Lambda}$ and the behavior of $E^{\Lambda}_i$,
$F^{\Lambda}_i$ on $\Proj(R^{\Lambda})$ and $\Rep(R^{\Lambda})$. In
particular, we will show that $E^{\Lambda}_i$ and $F^{\Lambda}_i$
are well-defined exact functors on $\Proj(R^{\Lambda})$ and
$\Rep(R^{\Lambda})$.

For $\Lambda \in \mathtt{P}^+$ and $i \in I$, we choose a monic polynomial of degree
$\langle h_i,\Lambda \rangle$
\begin{align} \label{Eq: cylotomic polynomial}
a^{\Lambda}_i(u)= \sum_{k=0}^{\langle h_i,\Lambda \rangle} c_{i;k}u^{\langle h_i,\Lambda \rangle-k}
\end{align}
with $c_{i;k} \in \k_{2kd_i}$ and $c_{i,0}=1$.

Given $\beta \in \mathtt{Q}^+$ with $|\beta|=n$, a dominant integral weight $\Lambda\in \mathtt{P}^{+}$ and
$k$ $(1 \le k \le n)$, set
$$ a^{\Lambda}(x_k) = \sum_{\nu \in I^\beta} a^{\Lambda}_{\nu_k}(x_k)e(\nu) \in R(\beta) .$$

\begin{definition} Let $\beta \in \mathtt{Q}^+$ and $\Lambda\in \mathtt{P}^{+}$.
\begin{enumerate}
\item The cyclotomic Khovanov-Lauda-Rouquier algebra $R^{\Lambda}(\beta)$ at $\beta$ is the quotient algebra
$$ R^{\Lambda}(\beta) = \dfrac{R(\beta)}{R(\beta)a^{\Lambda}(x_1)R(\beta)}.$$
\item The $\mathtt{Q}^+$-graded algebra $R^{\Lambda}= \bigoplus_{\alpha \in \mathtt{Q}^+} R^{\Lambda}(\alpha)$ is called
the {\it cyclotomic Khovanov-Lauda-Rouquier algebra} of weight $\Lambda$.
\end{enumerate}

\begin{lemma} \label{Lemma: zeros}
Let $\nu \in I^n$ be such that $\nu_a=\nu_{a+1}$ for some $1 \le a < n$. Then, for an $R(n)$-module
$M$ and $f \in \k[x_1,\ldots,x_n]$, $f e(\nu) M=0$ implies
\begin{align*}
 & (\partial_a f)\P_{\nu_a}(x_a,x_{a+1})\P_{\nu_a}(x_{a+1},x_a) e(\nu)M=0, \\
 & (s_a f)\P_{\nu_a}(x_a,x_{a+1})\P_{\nu_a}(x_{a+1},x_a) e(\nu)M=0.
\end{align*}
\end{lemma}

\begin{proof} Note that
$\tau_{a}e(\nu)=e(\nu)\tau_{a}$ and $\tau^2_{a}e(\nu)=(\partial_a \P_{\nu_a}(x_a,x_{a+1})) \tau_a e(\nu)$.
Thus we have
\begin{align*}
& \ (x_{a}-x_{a+1})\tau_{a}f\tau_{a} e(\nu) \\
& = (x_{a}-x_{a+1}) ( (s_a f) \tau_a + (\partial_a f)\P_{\nu_a}(x_a,x_{a+1}) ) \tau_a e(\nu) \\
& = (x_{a}-x_{a+1})\Bigl((\partial_a \P_{\nu_a}(x_a,x_{a+1})) (s_a f) +
    (\partial_a f)\P_{\nu_a}(x_a,x_{a+1})\Bigr)\tau_a e(\nu) \\
& = \bl(\P_{\nu_a}(x_{a+1},x_{a})-\P_{\nu_a}(x_a,x_{a+1})\br) (s_a f)\tau_a e(\nu) +
    \P_{\nu_a}(x_a,x_{a+1})(s_a(f)-f)\tau_a e(\nu) \\
&=\P_{\nu_a}(x_{a+1},x_{a})(s_a f)\tau_a e(\nu) -\P_{\nu_a}(x_a,x_{a+1})f\tau_a e(\nu) \\
& = \P_{\nu_a}(x_{a+1},x_a)\bl(\tau_a f-(\partial_a f)\P_{\nu_a}(x_a,x_{a+1})\br) e(\nu)
- \P_{\nu_a}(x_a,x_{a+1})f\tau_ae(\nu).
\end{align*}
Thus
$$(\partial_a f) \P_{\nu_a}(x_a,x_{a+1})\P_{\nu_a}(x_{a+1},x_a)e(\nu)M=0.$$
Since $(x_a-x_{a+1})(\partial_a f)=s_a f -f$, we have
$$(s_a f) \P_{\nu_a}(x_a,x_{a+1})\P_{\nu_a}(x_{a+1},x_a) e(\nu)M=0.$$
\end{proof}

\begin{lemma} \label{Lem: nilpotency}
 Let $\beta \in \mathtt{Q}^+$ with $|\beta|=n$.
\bnum
\item There exists a monic polynomial $g(u)$ such that $g(x_a)=0$ in $R^{\Lambda}(\beta)$ for
any $a$ $(1 \le a \le n)$.
\item If $i \in \Ire$, then there exists $m\in \Z_{\ge 0}$ such that $R^{\Lambda}(\beta+k\alpha_i)=0$
for any $k \ge m$.
\end{enumerate}
\end{lemma}

\begin{proof} (i) By induction on $a$, it is enough to show that
\begin{center}
For any monic polynomial $g(u)$, we can find
a monic polynomial $h(u)$ such that \\
we have $h(x_{a+1})M=0$ for any $R(\beta)$-module $M$ with $g(x_a)M=0$.
\end{center}
If $\nu_a=\nu_{a+1}$, then Lemma \ref{Lemma: zeros} implies that
$$ g(x_{a+1})\P_{\nu_a}(x_a,x_{a+1})\P_{\nu_a}(x_{a+1},x_a) e(\nu)M=0.$$
By the definition of $\P_i(u,v)$ given in
$\eqref{Def: P_i}$,
$(g(x_{a+1})\P_{\nu_a}(x_a,x_{a+1})\P_{\nu_a}(x_{a+1},x_a) )$
is a monic polynomial in $x_{a+1}$ with coefficient in $\k[x_a]$.
Hence we can choose a monic polynomial $h(x_{a+1})$ in the
ideal generated by $g(x_a)$ and $g(x_{a+1})\P_{\nu_a}(x_a,x_{a+1})\P_{\nu_a}(x_{a+1},x_a)$ in
$\k[x_a,x_{a+1}]$. Thus
$$h(x_{a+1})e(\nu)M=0.$$

 If $\nu_{a} \neq \nu_{a+1}$, then
$$g(x_{a+1}) \Q_{\nu_{a},\nu_{a+1}}(x_a,x_{a+1})e(\nu)M = g(x_{a+1})\tau^2_a e(\nu)M
= \tau_a g(x_a) e(s_a\nu)\tau_aM =0.$$
Since $g(x_{a+1})\Q_{\nu_{a},\nu_{a+1}}(x_a,x_{a+1})$ is a
monic polynomial in $x_{a+1}$ with coefficient in $\k[x_a]$,
we can choose a monic polynomial $h(x_{a+1})$ as in the case of
$\nu_{a}= \nu_{a+1}$.

\medskip\noi
(ii)\ For $\nu\in I^n$, set $\Supp_i(\nu)=\#\set{k}{\text{$1\le k\le n$ and $\nu_k=i$}}$.
 Our assertion is equivalent to:
\eq
&&\parbox{60ex}{For all $n$, there exists $k_n\in \Z_{\ge 0}$ such that
$e(\nu)R^{\Lambda}(n+ k_n)=0$ for
        any $\nu \in I^{n+ k_n}$ with $\Supp_i(\nu) \ge  k_n$.}
\label{eq:kn}
\eneq
If $e(\nu)R^{\Lambda}(n+k)=0$ for any $\nu\in I^{n+k}$ such that $\Supp_i(\nu)\ge k$,
then one can easily see that
\begin{align}\label{Eq: k' k}
\text{$e(\nu')R^{\Lambda}(n+k')=0$ for any $k' \ge k$ and $\nu'\in I^{n+k'}$ such that
$\Supp_i(\nu'_{\le n+k})\ge k$.}
\end{align}
In order to prove \eqref{eq:kn},  we will use induction on $n$.
Assume that there exists $k=k_{n-1}$ such that
$$e(\nu)R^{\Lambda}(n-1+k)=0\quad \text{if $\Supp_i(\nu)\ge k$.}$$
By (i), there exists a monic polynomial $g(u)$ of degree $m \ge 0$ such that
$g(x_{n+k})R^\Lambda(n+k)=0$.
It suffices to show
$$e(\nu)R^{\Lambda}(n+k+m)=0 \text{ for } \Supp_i(\nu) \ge k+m.$$
If $\Supp_i(\nu_{\le n+k-1})\ge k$, then by $\eqref{Eq: k' k}$
$e(\nu)R^{\Lambda}(n+k+m)=0$. Thus we may assume that $\Supp_i(\nu_{\le n+k-1})\le  k-1$.
Hence we have $\nu_{\ge n+k}=(i,\ldots,i)$. Then the repeated application of
Lemma \ref{Lemma: zeros} implies
$$(\partial_{n+k+m-1} \cdots\partial_{n+k}g(x_{n+k}))e(\nu)R^{\Lambda}(n+k+m)=0.$$
Since $\partial_{n+k+m-1} \cdots \partial_{n+k} g(x_{n+k})=\pm1$, we can
choose $k_n=k+m$.
\end{proof}

\begin{lemma} \label{Lem: imaginary integrable}
If $i \in \Iim$ and $\langle h_i, \Lambda-\beta \rangle=0$, then
$$R^{\Lambda}(\beta+\alpha_i)=0.$$
\end{lemma}

\begin{proof}
Since $\la h_i,\Lambda\ra$, $\la h_i,-\beta\ra\ge0$, the hypothesis
$\langle h_i, \Lambda-\beta \rangle=0$ implies $\langle
h_i, \Lambda \rangle =0$ and $\langle h_i, \beta \rangle=0$.
 Thus for all $j \in \Supp(\beta)\setminus\{i\}$, we have
$a_{ij}=0$. In particular we have $\Q_{j,i}\in\k_0^\times$. Since $\langle h_i, \Lambda \rangle=0$,
we have $e(i,\beta)R^{\Lambda}(\beta+\alpha_i)=0$.
For $\nu\in I^{\beta+\alpha_i}$, let $k$ be the smallest integer such that $\nu_k=i$.
We shall show $e(\nu)R^{\Lambda}(\beta+\alpha_i)=0$ by induction on $k$.
If $k$=1, it is obvious.
Assume $k>1$.
Hence  $\Q_{\nu_{k-1},\nu_{k}}e(\nu)R^{\Lambda}(\beta+\alpha_i)
=\tau_{ k-1}e(s_{k-1}\nu)\tau_{k-1}R^{\Lambda}(\beta+\alpha_i)$
vanishes since $(s_{k-1}\nu)_{k-1}=i$.
Since $\Q_{\nu_{k-1},\nu_{k}}\in\k_0^\times$, we obtain
the desired result $e(\nu)R^{\Lambda}(\beta+\alpha_i)=0$.
\end{proof}

For each $i \in I$, we define the functors
\begin{align*}
&E_i^{\Lambda}\cl \Mod(R^{\Lambda}(\beta+\alpha_i)) \to \Mod(R^{\Lambda}(\beta)),\\
&F_i^{\Lambda}\cl \Mod(R^{\Lambda}(\beta)) \to \Mod(R^{\Lambda}(\beta+\alpha_i)),
\end{align*}
by
\begin{align*}
&E_i^{\Lambda}(N)=e(\beta,i)N = e(\beta,i)R^{\Lambda}(\beta+\alpha_i) \otimes_{R^{\Lambda}(\beta+\alpha_i)}N, \\
&F_i^{\Lambda}(M)=R^{\Lambda}(\beta+\alpha_i)e(\beta,i)\otimes_{R^{\Lambda}(\beta)}M,
\end{align*}
where $M \in \Mod(R^{\Lambda}(\beta+\alpha_i))$ and $N \in \Mod(R^{\Lambda}(\beta))$.
\end{definition}

We introduce $\bl(R(\beta+\alpha_i),R^{\Lambda}(\beta)\br)$-bimodules
\begin{equation}\label{Def: Kernels}
\begin{aligned}
& F^{\Lambda}= R^{\Lambda}(\beta+\alpha_i)e(\beta,i)
= \dfrac{R(\beta+\alpha_i)e(\beta,i) }{R(\beta+\alpha_i)a^{\Lambda}(x_1)R(\beta+\alpha_i)e(\beta,i)}, \\
& K_0 = R(\beta+\alpha_i)e(\beta,i) \otimes_{R(\beta)} R^{\Lambda}(\beta)
= \dfrac{R(\beta+\alpha_i)e(\beta,i) }{R(\beta+\alpha_i)a^{\Lambda}(x_1)
R(\beta)e(\beta,i)}, \\
& K_1 = R(\beta+\alpha_i)e(i,\beta) \otimes_{R(\beta)} R^{\Lambda}(\beta)
= \dfrac{R(\beta+\alpha_i)e(i,\beta) }%
{R(\beta+\alpha_i)a^{\Lambda}(x_2)R^1(\beta)e(i,\beta)}.
\end{aligned}
\end{equation}

The right $R(\beta)$-module structure on
$R(\beta+\alpha_i)e(i,\beta)$ and the right
$R^{\Lambda}(\beta)$-module structure on $K_1$ are given by the
isomorphism $R(\beta) \overset{\sim} \to R^1(\beta) \hookrightarrow
R(\beta+\alpha_i)$. The bimodules $F^{\Lambda}$, $K_0$ and $K_1$ are the
kernels of the functors $F^{\Lambda}_i$, $F_i$ and $\overline{F}_i$
from $\Mod(R^{\Lambda}(\beta))$ to $\Mod(R(\beta+\alpha_i))$,
respectively.

Let $t_i$ be an indeterminate of degree $2d_i$. Then $\k[t_i]$ acts
from the right on $R(\beta+\alpha_i)e(i,\beta)$ and $K_1$ by
multiplying $x_1$. Similarly, $\k[t_i]$ acts from the right on
$R(\beta+\alpha_i)e(\beta,i)$, $F^{\Lambda}$ and $K_1$ by
multiplying $x_{n+1}$. Thus
 $K_0$, $F^{\Lambda}$ and $K_1$ have an $(R(\beta+\alpha_i),R^{\Lambda}(\beta)\otimes \k[t_i])$-bimodule
structure.

By  a similar argument to the one given
in \cite[Lemma 4.8, Lemma 4.16]{KK11},
we obtain the following lemmas which will be used in proving
Corollary~\ref{Cor: Exa F_i bar F_i} and Theorem~\ref{Thm: Exact}.
\begin{lemma} \label{Lem: Proj} \hfill
\bnum
\item
Both $K_1$ and $K_0$ are finitely generated projective right $R^{\Lambda}(\beta)\otimes \k[t_i]$-modules.
\item
In particular, for any $f(x_1,\ldots, x_{n+1})\in\k[x_1,\ldots, x_{n+1}]$
which is a monic polynomial in $x_1$, the right multiplication by $f$ on $K_1$
induces an injective endomorphism of $K_1$.
\ee
\end{lemma}

\begin{lemma} \label{Lem: Decom} For $i \in I$ and $\beta \in \mathtt{Q}^+$ with $|\beta|=n$, we have
\bnum
\item $R(\beta+\alpha_i)a^{\Lambda}(x_1)R(\beta+\alpha_i) =\sum_{a=0}^{n}R(\beta+\alpha_i)
a^{\Lambda}(x_1)\tau_1 \cdots \tau_a$,
\item
$R(\beta+\alpha_i)a^{\Lambda}(x_1)R(\beta+\alpha_i)e(\beta,i)$

$=R(\beta+\alpha_i)a^{\Lambda}(x_1)R(\beta)e(\beta,i)
+R(\beta+\alpha_i)a^{\Lambda}(x_1)\tau_1 \cdots
\tau_n e(\beta,i)$.
\end{enumerate}
\end{lemma}

Let $\pi\cl K_0 \to F^{\Lambda}$ be the canonical projection and
$\widetilde{P}\cl R(\beta+\alpha_i)e(i,\beta) \to K_0$ be the right multiplication by
$a^{\Lambda}(x_1)\tau_1 \cdots \tau_n$ whose degree is
$$ 2d_i \langle h_i,\Lambda \rangle + (\alpha_i| -\beta) = (\alpha_i| 2\Lambda-\beta).$$

Then, using Lemma \ref{Lem: Decom},
one can see that
\begin{align} \label{Eq: tilde P}
{\rm Im}(\widetilde{P}) = {\rm Ker} \pi
=\dfrac{R(\beta+\alpha_i) a^{\Lambda}(x_1)
R(\beta+\alpha_i)e(\beta,i)}{R(\beta+\alpha_i)a^{\Lambda}(x_1)R(\beta)e(\beta,i)}
\subset K_0.
\end{align}

\begin{lemma} \label{Lem: wide P, P}
The map $\widetilde{P}\cl R(\beta+\alpha_i)e(i,\beta) \to K_0 $ is a
right $R(\beta)\otimes \k[t_i]$-linear homomorphism; i.e., for all $S \in R(\beta+\alpha_i)$, $1 \le a \le n$
and $1 \le b \le n-1$,
\begin{align*}
\widetilde{P}(S x_{a+1}) = \widetilde{P}(S) x_{a}, \quad
\widetilde{P}(S x_{1}) = \widetilde{P}(S) x_{n+1}, \quad
\widetilde{P}(S \tau_{b+1}) = \widetilde{P}(S) \tau_{b}.
\end{align*}
\end{lemma}

\begin{proof}
First, we will verify that
\begin{align} \label{Eq: Mod RLambda}
&\parbox{77ex}{for any $f \in \k[x_1,\ldots,x_{n+1}]$,
$a^{\Lambda}(x_1)\tau_1 \cdots \tau_\ell f\tau_{k} \cdots \tau_n e(\beta,i)
\equiv 0 \mod R(\beta+\alpha_i)a^{\Lambda}(x_1)R(\beta)e(\beta,i)$ if
$\ell+2\le k\le n+1$.}
\end{align}
We will prove this by using downward induction on $k$. It is trivial for $k=n+1$.
Assume that $k\le n$ and our assertion is true for $k+1$. Then we have
\begin{equation} \label{Eq: -Ind k+1}
\begin{aligned}
a^{\Lambda}(x_1)\tau_1 \cdots \tau_\ell f \tau_k \cdots \tau_n e(\beta,i)& =
\tau_k a^{\Lambda}(x_1)\tau_1 \cdots \tau_\ell s_k(f)\tau_{k+1}
\cdots \tau_n e(\beta,i) \\
& \qquad \qquad \qquad +a^{\Lambda}(x_1)  \tau_1 \cdots \tau_\ell f'\tau_{k+1}
\cdots \tau_n e(\beta,i)
\end{aligned}
\end{equation}
for some $f' \in \k[x_1,\ldots,x_{n+1}]$, and
both the terms in the right-hand side of $\eqref{Eq: -Ind k+1}$ are 0
${\rm mod}\ R(\beta+\alpha_i) a^{\Lambda}(x_1)   R(\beta)  e(\beta,i) $
by the induction hypothesis.
Thus we obtain \eqref{Eq: Mod RLambda}.

\medskip
For $1 \le a \le n$, we have
\begin{align*}
x_{a+1}( a^{\Lambda}(x_1)   \tau_1 \cdots \tau_n  e(\beta,i) )
& =  a^{\Lambda}(x_1)  \tau_1 \cdots \tau_{a-1} (x_{a+1} \tau_{a}) \tau_{a+1} \cdots \tau_n
                                                                                     e(\beta,i) ,\\
& =  a^{\Lambda}(x_1)   \tau_1 \cdots \tau_{a-1}(\tau_{a} x_a + \P_{a,a+1})\tau_{a+1} \cdots \tau_n
                                                                                     e(\beta,i) , \\
& =  a^{\Lambda}(x_1)   \tau_1 \cdots \tau_n x_a  e(\beta,i)
   +  a^{\Lambda}(x_1)   \tau_1 \cdots \tau_{a-1} \P_{a,a+1}\tau_{a+1} \cdots \tau_n  e(\beta,i) ,\\
& \equiv  a^{\Lambda}(x_1)   \tau_1 \cdots \tau_n x_a  e(\beta,i)
                                                                \quad (\text{by } \eqref{Eq: Mod RLambda}).
\end{align*}

For the second assertion, we have
\begin{align*}
x_1( a^{\Lambda}(x_1)   \tau_1 \cdots \tau_n  e(\beta,i) )
&=  a^{\Lambda}(x_1)  (\tau_1 x_2 - \P_{1,2})\tau_2 \cdots \tau_n  e(\beta,i)  \\
&=  a^{\Lambda}(x_1)  \tau_1 x_2 \tau_2 \cdots \tau_n  e(\beta,i)
        - \P_{1,2}\tau_2 \cdots \tau_n  a^{\Lambda}(x_1)   e(\beta,i)  \\
& \equiv  a^{\Lambda}(x_1)  \tau_1 x_2 \tau_2 \cdots \tau_n  e(\beta,i)  \\
&=  a^{\Lambda}(x_1)  \tau_1 \tau_2 x_3 \tau_3 \cdots \tau_n  e(\beta,i)
 - a^{\Lambda}(x_1)  \tau_1 \P_{2,3} \tau_3 \cdots \tau_n  e(\beta,i)  \\
&\equiv  a^{\Lambda}(x_1)  \tau_1 \tau_2 x_3 \tau_3 \cdots \tau_n  e(\beta,i)
                                            \quad (\text{by } \eqref{Eq: Mod RLambda}) \\
& \quad \quad \vdots \\
& \equiv  a^{\Lambda}(x_1)  \tau_1 \cdots \tau_n x_{n+1}  e(\beta,i)
                    \quad {\rm mod} \ R(\beta+\alpha_i) a^{\Lambda}(x_1)   R(\beta) e(\beta,i) .
\end{align*}
For $1 \le b\le n-1$, we have
\begin{align*}
&\tau_{b+1}( a^{\Lambda}(x_1)   \tau_1 \cdots \tau_n  e(\beta,i) ) \\
&=  a^{\Lambda}(x_1)   \tau_{1}\cdots \tau_{b-1}(\tau_{b+1}\tau_b\tau_{b+1})\tau_{b+2} \cdots \tau_{n}
                                                                                 e(\beta,i)  \\
&=  a^{\Lambda}(x_1)   \tau_{1}\cdots \tau_{b-1}
(\tau_{b}\tau_{b+1}\tau_{b} + \overline{\Q}_b \P_{b,b+2} + \tau_b\overline{\P}'_b + \overline{\P}''_b\tau_{b+1})
\tau_{b+2}\cdots \tau_{n}  e(\beta,i)  \\
&=  a^{\Lambda}(x_1)   \tau_{1}\cdots \tau_{n}\tau_{b}  e(\beta,i) +
 a^{\Lambda}(x_1)   \tau_{1}\cdots \tau_{b-1}(\overline{\Q}_b \P_{b,b+2})\tau_{b+2}\cdots\tau_n  e(\beta,i)     \\
& \ \ +  a^{\Lambda}(x_1)   \tau_{1}\cdots \tau_{b}(\overline{\P}'_b)\tau_{b+2}\cdots\tau_n  e(\beta,i) +  a^{\Lambda}(x_1)   \tau_{1}\cdots \tau_{b-1}(\overline{\P}''_b)\tau_{b+1}\cdots\tau_n e(\beta,i) .
\end{align*}
By $\eqref{Eq: Mod RLambda}$, all the terms except the first one are
$0$ ${\rm mod} \ R(\beta+\alpha_i) a^{\Lambda}(x_1) R(\beta)e(\beta,i)$. Thus we
obtain
$$\tau_{b+1}a^{\Lambda}(x_1)\tau_1 \cdots \tau_n e(\beta,i)\equiv a^{\Lambda}(x_1)\cdots \tau_{n}\tau_{b}e(\beta,i)
\quad {\rm mod} \ R(\beta+\alpha_i)a^{\Lambda}(x_1)R(\beta) e(\beta,i).$$
\end{proof}

Since $\widetilde{P}$ is right $R(\beta) \otimes \k[t_i]$-linear and maps
$R(\beta+\alpha_i)a^{\Lambda}(x_2) R^1(\beta)e(i,\beta)$ to
$R(\beta+\alpha_i)a^{\Lambda}(x_1)R(\beta)e(\beta,i)$, it induces a map
$$P \cl K_1 \to K_0,$$
which is an $(R(\beta+\alpha_i),R(\beta) \otimes \k[t_i])$-bilinear
homomorphism. By $\eqref{Eq: tilde P}$, we get an exact sequence of
$(R(\beta+\alpha_i),R(\beta) \otimes \k[t_i])$-bimodules
$$ K_1 \overset{P}{\longrightarrow} K_0 \overset{\pi}{\longrightarrow} F^{\Lambda} \longrightarrow 0.$$
We will show that $P$ is actually injective by constructing an $(R(\beta+\alpha_i),R(\beta) \otimes
\k[t_i])$-bilinear homomorphism $Q$ such that $Q \circ P$ is injective.

For $1 \le a \le n$, we define an element $g_a$ of $R(\beta+\alpha_i)$ by
\begin{align}\label{Eq: g_a}
g_a =\sum_{ \substack{\nu \in I^{\beta+\alpha_i},\\ \nu_a \neq \nu_{a+1}}} \tau_a e(\nu)
+ \sum_{ \substack{\nu \in I^{\beta+\alpha_i},\\ \nu_a = \nu_{a+1}}}
 ((x_{a+1}-x_a)\P_{\nu_a}(x_a,x_{a+1})- (x_{a+1}-x_a)^2\tau_a)e(\nu).
\end{align}

\begin{lemma} For $1\le a \le n$, we have
\begin{align} \label{Eq: Commute g_a}
x_{s_a(b)}g_a = g_a x_b \ (1 \le b \le n+1) \quad \text{ and } \quad
\tau_a g_{a+1} g_a =  g_{a+1} g_a \tau_{a+1}.
\end{align}
\end{lemma}

\begin{proof} For $\nu$ such that $\nu_a \neq \nu_{a+1}$, we have
\eq&&x_{s_a(b)}g_a e(\nu)= g_a x_be(\nu). \label{eq:g}
\eneq
We shall show \eqref{eq:g} when $\nu_a =\nu_{a+1}$. We have
\begin{align*}
&(x_ag_a -g_a x_{a+1})e(\nu) \\
&= \{  x_a(x_{a+1}-x_a)\P_{\nu_a}(x_a,x_{a+1})-x_a(x_{a+1}-x_a)^2\tau_a
-(x_{a+1}-x_a)x_{a+1}\P_{\nu_a}(x_a,x_{a+1})  \\
&  \quad \quad + (x_{a+1}-x_a)^2 (x_a \tau_a + \P_{\nu_a}(x_a,x_{a+1}))e(\nu) \} \\
&= \{ -(x_{a+1}-x_a)^2\P_{\nu_a}(x_a,x_{a+1})+ (x_{a+1}-x_a)^2\P_{\nu_a}(x_a,x_{a+1}) \} e(\nu) =0.
\end{align*}
Hence \eqref{eq:g} holds for $b=a+1$. The other cases can be proved similarly.

\smallskip
By $\eqref{Eq: Braid Relation}$, $S=\tau_ag_{a+1}g_a -g_{a+1}g_a
\tau_{a+1}$ does not contain the term $\tau_{a+1}\tau_a\tau_{a+1}$
and $\tau_a\tau_{a+1}\tau_a$ and is contained in the
$\k[x_a,x_{a+1},x_{a+2}]$-module generated by $1$, $\tau_a$,
$\tau_{a+1}$, $\tau_a\tau_{a+1}$, $\tau_{a+1}\tau_a$. That is, $S$
can be expressed as
$$ S=\mathsf{T}_1+\mathsf{T}_2\tau_a+\mathsf{T}_3\tau_{a+1}+\mathsf{T}_4\tau_a\tau_{a+1}
+\mathsf{T}_5\tau_{a+1}\tau_a$$ for some $\mathsf{T}_i \in
\k[x_a,x_{a+1},x_{a+2}] \ (1 \le i \le 5)$. By a similar argument
given in  \cite[Lemma 4.12]{KK11}, we have
$$ Sx_{b} = x_{s_{a,a+2}(b)}S  \quad \text{ for all } b.$$
Then one can show that all $\mathsf{T}_i$ must be zero. Thus our second assertion holds.
\end{proof}

\begin{proposition} \
\bnum
\item Let $\widetilde{Q}\cl R(\beta+\alpha_i)e(\beta,i) \to K_1$ be the
 left
$R(\beta+\alpha_i)$-linear homomorphism given by the multiplication of $g_n \cdots g_1$
from the right. Then
$\widetilde{Q}$ is  a right $\bl(R(\beta)\otimes k[t_i]\br)$-linear homomorphism. That is,
\begin{align*}
& \widetilde{Q}(S x_a)= \widetilde{Q}(S) x_{a+1} \ (1 \le a \le n), \quad
   \widetilde{Q}(S x_{n+1})=\widetilde{Q}(S)x_1 \\
&   \widetilde{Q}(S \tau_b)= \widetilde{Q}(S) \tau_{b+1} \ (1 \le b
\le n-1)
\end{align*}
for any $S \in R(\beta+\alpha_i)e(\beta,i)$.
\item The map $\widetilde{Q}$ induces a well-defined $(R(\beta+\alpha_i),R(\beta)\otimes k[t_i])$-bilinear
homomorphism
$$Q\cl K_0 = \dfrac{R(\beta+\alpha_i)e(\beta,i) }{R(\beta+\alpha_i)
a^{\Lambda}(x_1)R(\beta)e(\beta,i)} \to
       K_1 = \dfrac{R(\beta+\alpha_i)e(i,\beta) }{R(\beta+\alpha_i)
a^{\Lambda}(x_2)R^1(\beta)e(i,\beta)}.$$
\end{enumerate}
\end{proposition}

\begin{proof} The proof follows immediately from the preceding lemma.
\end{proof}

\begin{theorem} \label{Thm: A nu}
For each $\nu \in I^{\beta}$, set

$$ \mathsf{A}_{\nu} = a^{\Lambda}_i(x_1)
\prod_{\substack{1 \le a \le n, \\ \nu_a \neq
i}}\Q_{i,\nu_a}(x_1,x_{a+1}) \prod_{\substack{1 \le a \le n, \\
\nu_a = i}}\P_i(x_1,x_{a+1})\P_i(x_{a+1},x_1). $$

Then the following diagram is
commutative, in which the vertical arrow is the multiplication by
$\mathsf{A}_{\nu}$ from the right.
\begin{align} \label{Dia: inj}
 \xymatrix{
\dfrac{R(\beta+\alpha_i)e(i,\nu) }{R(\beta+\alpha_i)a^{\Lambda}(x_2)R^1(\beta)e(i,\nu)}
\ar[rr]^{P=a^{\Lambda}(x_1) \tau_1 \cdots \tau_n} \ar[d]_{\mathsf{A}_{\nu}}& &
\dfrac{R(\beta+\alpha_i)e(\nu,i) }{R(\beta+\alpha_i)a^{\Lambda}(x_1)R(\beta)e(\nu,i)}
\ar[dll]^{Q=g_n \cdots g_1} \\
\dfrac{R(\beta+\alpha_i)e(i,\nu) }{R(\beta+\alpha_i)a^{\Lambda}(x_2)R^1(\beta)e(i,\nu)}
}
\end{align}
\end{theorem}

\begin{proof}
It suffices to show that
\begin{equation} \label{Eq: claim1}
\begin{aligned}
a^{\Lambda}(x_1)\tau_1 \cdots \tau_n g_n \cdots g_1 e(i,\nu)
&= a^{\Lambda}(x_1)\tau_1 \cdots \tau_n e(\nu,i) g_n \cdots g_1  \\
& \equiv \mathsf{A}_{\nu}e(i,\nu) \quad {\rm mod} \ R(\beta+\alpha_i)
a^{\Lambda}(x_2) R^1(\beta)e(i,\nu).
\end{aligned}
\end{equation}
Note that
\begin{equation} \label{Eq: tau n g n}
\begin{aligned}
\tau_n e(\nu,i) g_n=
\left\{
\begin{array}{ll}
   \tau_n e(\nu,i)\tau_n = \Q_{i,\nu_n}(x_n,x_{n+1})e(\nu_{<n},i,\nu_n) &  \text{ if } \nu_n \neq i,\\
 \tau_n(x_{n+1}-x_n)\P_i(x_{n+1},x_{n}) e(\nu,i) &  \text{ if } \nu_n = i.  \\
\end{array}
\right.
\end{aligned}
\end{equation}

Indeed if $\nu_n=i$, then we have
\eqn
&&\hs{5ex}\tau_n \Bigl((x_{n+1}-x_n)\P_i(x_n,x_{n+1})- (x_{n+1}-x_n)^2\tau_n\Bigr) e(\nu,i)\\
&&\hs{5ex}=\Bigl(\tau_n (x_{n+1}-x_n)\P_i(x_n,x_{n+1})- \tau_n^2(x_{n+1}-x_n)^2
\Bigr) e(\nu,i)\\
&&\hs{5ex}=\Bigl(\tau_n (x_{n+1}-x_n)\P_i(x_n,x_{n+1})
- \tau_n \bl(\partial_n\P_i(x_n,x_{n+1})\br)(x_{n+1}-x_n)^2)\Bigr) e(\nu,i)\\
&&\hs{5ex}=\tau_n (x_{n+1}-x_n)\Bigl(\P_i(x_n,x_{n+1})-
\bl(\P_i(x_n,x_{n+1})-\P_i(x_{n+1},x_{n})\br)\Bigr) e(\nu,i)\\
&&\hs{5ex}=\tau_n (x_{n+1}-x_n)\P_i(x_{n+1},x_{n})e(\nu,i).
\eneqn

We will show \eqref{Eq: claim1} by induction on $n$. Assume first $n=1$. If $\nu_1 \neq i$, then it is already given by
$\eqref{Eq: tau n g n}$. If $\nu_1 = i$, then
\begin{align*}
a^{\Lambda}(x_1)\tau_1 e(i,i)g_1 & = a^{\Lambda}(x_1)\tau_1(x_2-x_1)\P_i(x_2,x_1)
e(i,i) \\
&= \bl(\tau_1 a^{\Lambda}(x_2)+ \dfrac{a^{\Lambda}(x_2) -a^{\Lambda}(x_1)}{x_1-x_2}\P_i(x_1,x_2)\br)(x_2-x_1)\P_i(x_2,x_1) e(i,i)\\
&= \bl(\tau_1 a^{\Lambda}(x_2)(x_2-x_1)\P_i(x_2,x_1)
- (a^{\Lambda}(x_2)-a^{\Lambda}(x_1))P_i(x_1,x_2)\P_i(x_2,x_1)\br)e(i,i) \\
& \equiv a^{\Lambda}(x_1)\P_i(x_1,x_2)\P_i(x_2,x_1) e(i,i) = \mathsf{A}_{\nu}e(i,i).
\end{align*}
Thus we may assume that $n >1$.

(i) First assume that $\nu_n \neq i$. Then we have
\begin{align*}
& a^{\Lambda}(x_1)\tau_1 \cdots \tau_n g_n \cdots g_1 e(i,\nu) \\
& =a^{\Lambda}(x_1)\tau_1 \cdots \tau_{n-1}\Q_{i,\nu_n}(x_n,x_{n+1})g_{n-1} \cdots g_1 e(i,\nu) \\
& =a^{\Lambda}(x_1)\tau_1 \cdots \tau_{n-1}g_{n-1} \cdots g_1 \Q_{i,\nu_n}(x_1,x_{n+1})e(i,\nu) \\
& \equiv \mathsf{A}_{\nu_{<n}}\Q_{i,\nu_n}(x_1,x_{n+1})e(i,\nu) = \mathsf{A}_{\nu}e(i,\nu).
\end{align*}

(ii) If $\nu_n = i$, then we have
\begin{equation} \label{Eq ; nu n i}
\begin{aligned}
& a^{\Lambda}(x_1)\tau_1 \cdots \tau_n g_n \cdots g_1 e(i,\nu) \\
&=a^{\Lambda}(x_1)\tau_1 \cdots \tau_n (x_{n+1}-x_n)\P_i(x_{n+1},x_{n})
g_{n-1} \cdots g_1 e(i,\nu) \\
&=a^{\Lambda}(x_1)\tau_1 \cdots \tau_n g_{n-1} \cdots g_1 (x_{n+1}-x_1)
\P_i(x_{n+1},x_{1}) e(i,\nu).
\end{aligned}
\end{equation}
Since $P$ and $Q$ are right $R(\beta) \otimes \k[t_i]$-linear, we have
\begin{equation} \label{Eq: commute x n+1}
\begin{aligned}
\ba{l}
x_{n+1}a^{\Lambda}(x_1)\tau_1 \cdots \tau_n g_n\cdots g_1  e(i,\nu)-
a^{\Lambda}(x_1)\tau_1 \cdots \tau_n g_n \cdots g_1 x_{n+1} e(i,\nu)\equiv0\\
\hs{40ex}\mod R(\beta+\alpha_i)a^{\Lambda}(x_2)R^1(\beta)e(i,\beta).
\ea
\end{aligned}
\end{equation}
By $\eqref{Eq ; nu n i}$, the left-hand side of $\eqref{Eq: commute
x n+1}$ is equal to
\begin{align} \label{Eq: commute setting}
a^{\Lambda}(x_1)\tau_1 \cdots \tau_{n-1}(x_{n+1}\tau_n-\tau_n x_{n+1})g_{n-1} \cdots g_1
(x_{n+1}-x_1)\P_i(x_{n+1},x_{1}) e(i,\nu).
\end{align}
Since
\begin{align*}
(x_{n+1}\tau_n-\tau_n x_{n+1})e(\nu,i) & = \{ (x_{n+1}\tau_n-\tau_n x_{n})+\tau_n(x_n-x_{n+1}) \} e(i,\nu) \\
&= \P_i(x_n,x_{n+1})+\tau_n(x_n-x_{n+1}),
\end{align*}
we have
\begin{equation} \label{Eq: reexpression}
\begin{aligned}
0&\equiv
a^{\Lambda}(x_1)\tau_1 \cdots \tau_{n-1}\P_i(x_n,x_{n+1}) g_{n-1} \cdots g_1
        (x_{n+1}-x_1)\P_i(x_{n+1},x_{1}) e(i,\nu) \\
& \qquad + a^{\Lambda}(x_1)\tau_1 \cdots \tau_{n}(x_n-x_{n+1}) g_{n-1} \cdots g_1
        (x_{n+1}-x_1)\P_i(x_{n+1},x_{1}) e(i,\nu) \\
& = a^{\Lambda}(x_1)\tau_1 \cdots \tau_{n-1} g_{n-1} \cdots g_1 (x_{n+1}-x_1)
\P_i(x_1,x_{n+1})\P_i(x_{n+1},x_{1}) e(i,\nu) \\
& \qquad -a^{\Lambda}(x_1) \tau_1 \cdots \tau_ng_{n-1} \cdots g_1 (x_{n+1}-x_1)^2
\P_i(x_{n+1},x_{1}) e(i,\nu) \\
&\equiv \mathsf{A}_{\nu_{<n}}(x_{n+1}-x_1)\P_i(x_1,x_{n+1})\P_i(x_{n+1},x_{1})
e(i,\nu)  \\
& \qquad -a^{\Lambda}(x_1)\tau_1 \cdots \tau_ng_{n-1} \cdots g_1 (x_{n+1}-x_1)^2
\P_i(x_{n+1},x_{1}) e(i,\nu)  \\
&=\Bigl(\mathsf{A}_{\nu_{<n}}\P_i(x_1,x_{n+1})
-a^{\Lambda}(x_1)\tau_1 \cdots \tau_ng_{n-1} \cdots g_1 (x_{n+1}-x_1)\Bigr)e(i,\nu)\\
&\hs{33ex}\times(x_{n+1}-x_1)\P_i(x_{n+1},x_{1}).
\end{aligned}
\end{equation}
Since the right multiplication of $(x_{n+1}-x_1)\P_i(x_{n+1},x_{1})$
on $K_1$ is injective by Lemma~\ref{Lem: Proj}, we conclude that
$$a^{\Lambda}(x_1)\tau_1 \cdots \tau_ng_{n-1} \cdots g_1 (x_{n+1}-x_1)e(i,\nu) \equiv
   \mathsf{A}_{\nu_{<n}}\P_i(x_1,x_{n+1})e(i,\nu).$$
Hence \eqref{Eq ; nu n i} implies that
\begin{align*}
a^{\Lambda}(x_1)\tau_1 \cdots \tau_n g_n \cdots g_1 e(i,\nu) & \equiv \mathsf{A}_{\nu_{<n}}
\P_i(x_1,x_{n+1})\P_i(x_{n+1},x_{1}) e(i,\nu)\\
& = \mathsf{A}_{\nu}e(i,\nu) \quad {\rm mod} \ R(\beta+\alpha_i)
a^{\Lambda}(x_2)R^1(\beta)e(i,\beta).
\end{align*}
\end{proof}

Since $K_1e(i,\nu)$ is a projective $R^{\Lambda}(\beta)\otimes
\k[t_i]$-module by Lemma \ref{Lem: Proj} and $\mathsf{A}_\nu$ is a
monic polynomial (up to a multiple of an invertible element) in
$t_i$, by  a similar argument to the one  in \cite[Lemma 4.17, Lemma 4.18]{KK11},
we conclude:

\begin{theorem} \label{Thm: P-injective}
We have a short exact sequence consisting of right projective $R^{\Lambda}(\beta)$-modules:
\begin{align}\label{Eq: The exact seq}
0 \to K_1 \xrightarrow{\;P\;} K_0 \to F^{\Lambda} \to 0.
\end{align}
\end{theorem}

Since $K_1$, $K_0$ and $F^{\Lambda}$ are kernels of functors $\overline{F}_i$, $F_i$ and $F^{\Lambda}_i$,
respectively, we have

\begin{corollary} \label{Cor: Exa F_i bar F_i}
 For any $i \in I$ and $\beta \in \mathtt{Q}^+$, there exists an exact sequence of
$R(\beta+\alpha_i)$-modules
\begin{align} \label{Eq: Fi bar Fi Fi Lam}
0 \to q^{(\alpha_i|2\Lambda-\beta)}\overline{F}_i M \to F_i M \to F^{\Lambda}_i M \to 0,
\end{align}
which is functorial in $M \in \Mod (R^{\Lambda}(\beta))$.
\end{corollary}

Now we prove the main theorem of this section.
\begin{theorem} \label{Thm: Exact}
Set
\begin{align*}
\Proj(R^{\Lambda}) = \bigoplus_{\alpha \in \mathtt{Q}^+}\Proj(R^{\Lambda}(\alpha)), \quad
\Rep(R^{\Lambda}) = \bigoplus_{\alpha \in \mathtt{Q}^+}\Rep(R^{\Lambda}(\alpha)),\quad\text{and}\\
[\Proj(R^{\Lambda})] = \bigoplus_{\alpha \in \mathtt{Q}^+}[\Proj(R^{\Lambda}(\alpha))], \quad
[\Rep(R^{\Lambda})] = \bigoplus_{\alpha \in \mathtt{Q}^+}[\Rep(R^{\Lambda}(\alpha))].
\end{align*}
 Then the functors $E^{\Lambda}_i$ and $F^{\Lambda}_i$ are well-defined exact functors on
$\Proj(R^{\Lambda})$  and $\Rep(R^{\Lambda})$, and they induce endomorphisms
of the Grothendieck groups $[\Proj(R^{\Lambda})]$ and $[\Rep(R^{\Lambda})]$.
\end{theorem}

\begin{proof}
By Proposition \ref{Prop: essence of KOP1}, Lemma
\ref{Lem: nilpotency} and Theorem \ref{Thm: P-injective},
$F^{\Lambda}$ is a finitely generated projective
module over $R^{\Lambda}(\beta)$.
Thus $F^\Lambda_i$ sends the finite-dimensional left
$R^\Lambda(\beta)$-modules to finite-dimensional
left $R^\Lambda(\beta+\alpha_i)$-modules. Similarly,
$e(\beta,i)R^\Lambda(\beta+\alpha_i)$ is a finitely generated
projective left $R^\Lambda(\beta)$-module and hence $E^\Lambda_i$
sends finitely generated projective left
$R^{\Lambda}(\beta+\alpha_i)$-modules to finitely generated
projective left $R^{\Lambda}(\beta)$-modules.
\end{proof}

The following lemma will be needed in the sequel.

\begin{lemma} \label{Lem : B nu}
Set
\begin{itemize}
\item $\mathsf{A}=\sum_{\nu \in I^{\beta}} \mathsf{A}_{\nu}e(i,\nu),$
\item $\mathsf{B}= \sum_{\nu \in I^{\beta}} a^{\Lambda}_i(x_{n+1})
\prod_{\substack{1 \le a \le n, \\ \nu_a \neq i}}\Q_{i,\nu_a}(x_{n+1},x_{a})
\prod_{\substack{1 \le a \le n, \\ \nu_a = i}}\P_i(x_{n+1},x_{a})\P_i(x_a,x_{n+1})e(\nu,i)$.
\end{itemize}
Then we have a commutative diagram
$$
\xymatrix{ K_1 \ar[r]^P \ar[d]_{\mathsf{A}}& K_0 \ar[d]^{\mathsf{B}} \ar[dl]|{Q}
           \\ K_1 \ar[r]_P & K_0}
$$
Here the vertical arrows are the multiplication of $\mathsf{A}$ and
$\mathsf{B}$ from the right,
respectively.
\end{lemma}

\begin{proof} We can apply a similar argument given in \cite[Lemma 4.19]{KK11}.
\end{proof}

\vskip 2em

\section{ Categorification of $V(\Lambda)$} \label{Sec: Categorification}

In this chapter, we will show that the cyclotomic
Khovanov-Lauda-Rouquier algebra $R^{\Lambda}$ categorifies the
irreducible highest weight $U_q(\g)$-module $V(\Lambda)$.

\begin{theorem}
For $i \neq j \in I$, there exists a natural isomorphism
\begin{align} \label{Thm: Comm E_i Lam F_j Lam}
 E^{\Lambda}_iF^{\Lambda}_j \simeq q_i^{-a_{ij}}F^{\Lambda}_jE^{\Lambda}_i.
\end{align}
\end{theorem}

\begin{proof}
By Corollary \ref{Cor: commutiation E_iF_j}, we already know
\begin{align} \label{Eq: commutiation E_iF_j}
e(n,i)R(n+1)e(n,j) \simeq q_i^{-a_{ij}}R(n)e(n-1,j)\otimes_{R(n-1)}e(n-1,j)R(n).
\end{align}
Applying the functor $R^{\Lambda}(n)\otimes_{R(n)} \ \bullet \ \otimes_{R(n)}R^{\Lambda}(n)e(\beta)$ on
$\eqref{Eq: commutiation E_iF_j}$, we obtain
\begin{align*}
& \dfrac{e(n,i)R(n+1)e(\beta,j)}{e(n,i)R(n)a^{\Lambda}(x_1) R(n+1)e(\beta,j)+e(n,i)R(n+1)a^{\Lambda}(x_1)R(n)e(\beta,j)}
 \\
&\hs{10ex}\simeq R^{\Lambda}(n)e(n-1,j)\otimes_{R^{\Lambda}(n-1)}e(n-1,i)R^{\Lambda}(n)e(\beta)
= F^{\Lambda}_jE^{\Lambda}_iR^{\Lambda}(\beta).
\end{align*}

Note that
\begin{align*}
E^{\Lambda}_iF^{\Lambda}_jR^{\Lambda}(\beta)
=\left( \dfrac{e(n,i)R(n+1)e(n,j)}{e(n,i)R(n+1)a^{\Lambda}(x_1)R(n+1)e(n,j)} \right) e(\beta).
\end{align*}
Thus it suffices to show that
\begin{equation}\label{Eq: claim Comm E_i Lam F_j Lam}
\begin{aligned}
& e(n,i)R(n+1)a^{\Lambda}(x_1)R(n+1)e(n,j) \\
& =e(n,i)R(n)a^{\Lambda}(x_1)R(n+1)e(n,j)+e(n,i)R(n+1)a^{\Lambda}(x_1)R(n)e(n,j).
\end{aligned}
\end{equation}

Since $a^{\Lambda}(x_1)\tau_{k} = \tau_{k}a^{\Lambda}(x_1)$ for all $k \ge 2$, we have
\begin{align*}
& \ R(n+1)a^{\Lambda}(x_1)R(n+1) = \sum_{a=1}^{n+1}R(n+1)a^{\Lambda}(x_1)\tau_a \cdots \tau_n R(n,1) \\
& = R(n+1)a^{\Lambda}(x_1)R(n,1)+R(n+1) a^{\Lambda}(x_1)
\tau_1 \cdots \tau_n R(n,1) \\
& = R(n+1)a^{\Lambda}(x_1)R(n,1)+\sum_{a=1}^{n+1}R(n,1)\tau_n \cdots \tau_a
a^{\Lambda}(x_1)\tau_1 \cdots \tau_n
    R(n,1) \\
& = R(n+1)a^{\Lambda}(x_1)R(n,1)+R(n,1)a^{\Lambda}(x_1)R(n+1)+  R(n,1)\tau_n \cdots \tau_1 a^{\Lambda}(x_1)\tau_1
\cdots \tau_n R(n,1).
\end{align*}

For $i \neq j$, we get
$$ e(n,i)R(n,1)\tau_n \cdots \tau_1 a^{\Lambda}(x_1)\tau_1 \cdots \tau_n R(n,1) e(n,j) =0,$$
and our assertion $\eqref{Eq: claim Comm E_i Lam F_j Lam}$ follows.
\end{proof}

\begin{theorem} \label{Thm: Main}
Let $\lambda=\Lambda-\beta$. Then there exist natural isomorphisms
of endofunctors on $\Mod(R^{\Lambda}(\beta))$ given below. \bnum
\item If $\langle h_i,\lambda \rangle \ge 0$, then we have
\begin{align} \label{Eq: Comm E_i Lam F_i Lam 1}
q_i^{-a_{ii}}F^{\Lambda}_iE^{\Lambda}_i \oplus
\bigoplus^{\langle h_i,\lambda \rangle-1}_{k=0} q_i^{2k}{\rm Id} \overset{\sim}{\to}
E^{\Lambda}_iF^{\Lambda}_i.
\end{align}
\item If $\langle h_i,\lambda \rangle < 0$, then we have
\begin{align} \label{Eq: Comm E_i Lam F_i Lam 2}
q_i^{-a_{ii}}F^{\Lambda}_iE^{\Lambda}_i \overset{\sim}{\to}
E^{\Lambda}_iF^{\Lambda}_i \oplus \bigoplus^{-\langle h_i,\lambda \rangle-1}_{k=0} q_i^{2k-2}{\rm Id}.
\end{align}
\end{enumerate}
\end{theorem}

The rest of this section is devoted to the proof of this theorem.

\bigskip
Consider the following commutative diagram with exact rows and
columns derived from Theorem \ref{Thm: Comm E_i F_j}, Theorem
\ref{Thm: Comm E_i bar F_j } and Corollary \ref{Cor: Exa F_i bar F_i}:
\begin{equation} \label{Dia: Comm at Mod}
\begin{aligned}
\xymatrix
{ & 0 \ar[d] & 0 \ar[d]&  &  & \\
 0 \ar[r] & q_i^{(\alpha_i|2\Lambda-\beta)}\overline{F}_iE_i M \ar[d] \ar[r]
 & q_i^{-a_{ii}}F_iE_i M \ar[d]\ar[r] & q_i^{-a_{ii}} F^{\Lambda}_iE^{\Lambda}_iM \ar[d]\ar[r] & 0 \\
 0 \ar[r] & q_i^{(\alpha_i|2\Lambda-\beta)}E_i\overline{F}_i M \ar[d] \ar[r]
 & E_iF_i M \ar[d]\ar[r] & E^{\Lambda}_i F^{\Lambda}_i M \ar[r] & 0 \\
 & q_i^{(\alpha_i|2\Lambda-2\beta)}\k[t_i]\otimes M \ar[d]\ar[r] & \k[t_i]\otimes M \ar[d] & & \\
 & 0 & 0 & &
}
\end{aligned}
\end{equation}

By taking the kernel modules, we obtain the following commutative diagram of $(R(\beta),R^\Lambda(\beta))$-modules:

\begin{equation} \label{Dia: Comm at Ker}
\begin{aligned}
\xymatrix@C=3ex
{ & 0 \ar[d] & 0 \ar[d]&  &  & \\
 0 \ar[r] & q_i^{(\alpha_i|2\Lambda-\beta)}K'_1 \ar[d] \ar[r]^-{P'}
 & q_i^{-a_{ii}}K'_0 \ar[d]^F \ar[r]^-{G}
 & q_i^{-a_{ii}} F^{\Lambda}_iE^{\Lambda}_iR^{\Lambda}(\beta) \ar[d]\ar[r] & 0 \\
 0 \ar[r] & q_i^{(\alpha_i|2\Lambda-\beta)}E_i K_1 \ar[d]_B \ar[r]^-{P}
 & E_i K_0 \ar[d]^C\ar[r] & E^{\Lambda}_i F^{\Lambda}_i R^{\Lambda}(\beta) \ar[r] & 0 \\
 & q_i^{(\alpha_i|2\Lambda-2\beta)}\k[t_i]\otimes R^{\Lambda}(\beta) \ar[d]\ar[r]^-{A}
 & \k[t_i]\otimes R^{\Lambda}(\beta) \ar[d] & & \\
 & 0 & 0 & &
}
\end{aligned}
\end{equation}
where
\begin{align*}
 K'_0& =F_iE_iR^{\Lambda}(\beta)=R(\beta)e(\beta-\alpha_i,i)\otimes_{R(\beta-\alpha_i)}e(\beta-\alpha_i)
       R^{\Lambda}(\beta) \\
 K'_1& =\overline{F}_iE_iR^{\Lambda}(\beta)= R(\beta)e(i,\beta-\alpha_i)\otimes_{R(\beta-\alpha_i)}
        e(\beta-\alpha_i)R^1(\beta) \otimes_{R(\beta)}R^{\Lambda}(\beta) \\
     & = R(\beta)e(i,\beta-\alpha_i)\otimes_{R(\beta-\alpha_i)}R^{\Lambda}(\beta).
\end{align*}
The homomorphisms in the diagram $\eqref{Dia: Comm at Ker}$ can be described as follows:
\begin{itemize}
\item $P$ is the right multiplication by $a^{\Lambda}(x_1)\tau_1 \cdots \tau_{n}$ and
      $(R(\beta),R^{\Lambda}(\beta)\otimes\k[t_i])$-bilinear.
\item Similarly, $P'$ is given by the right multiplication by $a^{\Lambda}(x_1)\tau_1 \cdots \tau_{n-1}$ on
      $R(\beta)e(i,\beta-\alpha_i)$.
\item The map $A$ is defined by the chasing the diagram. Note that it is $R^{\Lambda}(\beta)$-linear
      but {\it not\/} $\k[t_i]$-linear.
\item $B$ is given by taking the coefficient of $\tau_n \cdots \tau_1$ and
      $(R(\beta)\otimes\k[x_{n+1}],\k[x_1]\otimes R^1(\beta))$-bilinear.
\item $F$ is the multiplication by $\tau_n$ (See Proposition \ref{Prop: twist by tau n}).
\item $C$ is the cokernel map of $F$. Thus it is $(R(\beta),R^{\Lambda}(\beta))$-bilinear but does {\it not} commute with $t_i$.
\item $G$ is the canonical projection induced from $P'$. It is
$(R(\beta)\otimes \k[x_{n+1}],R^\Lambda(\beta)\otimes \k[x_{n+1}])$-bilinear.
\end{itemize}

Set $p = \Supp_i(\beta)$.
Note that the degree of $t_i$ in
$$
\displaystyle \prod_{\substack{1 \le a \le n, \\ \nu_a \neq i}}\Q_{i,\nu_a}(t_i,x_{a})
\prod_{\substack{1 \le a \le n, \\ \nu_a = i}}\P_i(t_i,x_{a+1})\P_i(x_{a+1},t_i)$$
is given by
\begin{align*}
 -\langle h_i,\beta-p \alpha_i \rangle + 2p(1 - \dfrac{a_{ii}}{2})
 = -\langle h_i,\beta \rangle+ p a_{ii} + 2p - p a_{ii} = -\langle h_i,\beta \rangle+2p.
\end{align*}

Define an invertible element $\gamma \in \k^{\times}$ by
\begin{equation} \label{Eq: def gamma}
\begin{aligned}
& (-1)^{p}\prod_{\substack{1 \le a \le n, \\ \nu_a \neq i}}\Q_{i,\nu_a}(t_i,x_{a})
\prod_{\substack{1 \le a \le n, \\ \nu_a = i}} \P_i(t_i,x_{a+1})\P_i(x_{a+1},t_i)\\
& = \gamma^{-1}t^{-\langle h_i,\beta \rangle+2p}_i +
   \left( \text{ terms of degree } < -\langle h_i,\beta \rangle+2p \text{ in } t_i \right).
\end{aligned}
\end{equation}

Set $\lambda=\Lambda - \beta$ and
\begin{align} \label{Eq: varphi k}
\varphi_k = A(t^k_i) \in \k[t_i] \otimes R^{\Lambda}(\beta),
\end{align}
which is of degree $2(\alpha_i|\lambda)+2d_i k =2d_i(\langle h_i,\lambda\rangle+k)$.

The following proposition is one of the key ingredients of the proof of Theorem \ref{Thm: Main}

\begin{proposition} \label{Prop: varphi k}
If $\langle h_i,\lambda\rangle+k \ge 0$, then $\gamma\varphi_k$ is a monic polynomial in $t_i$ of degree
$\langle h_i,\lambda\rangle+k$.
\end{proposition}
Note that for $m<0$, we say that a polynomial $\varphi$ is a monic polynomial of degree $m$
if $\varphi=0$.

To prove Proposition \ref{Prop: varphi k}, we need some preparation.
Let
$$z = \sum_{k \in \Z_{>0}} a_k \otimes b_k\in R(\beta)e(\beta-\alpha_i,i) \otimes_{R(\beta-\alpha_i)}
e(\beta-\alpha_i,i)R^{\Lambda} (\beta), $$ where
$a_k \in R(\beta)e(\beta-\alpha_i,i)$ and $b_k \in e(\beta-\alpha_i,i)R^{\Lambda}(\beta)$.
Define a map $E\cl K'_0 \to E_i K_0$ by
\begin{align} \label{Eq: Def F}
z \mapsto \sum_{k \in \Z_{>0}} a_k \P_i(x_n,x_{n+1})b_k.
\end{align}

\begin{lemma} For $z \in R(\beta)e(\beta-\alpha_i) \otimes_{R(\beta-\alpha_i)} e(\beta-\alpha_i,i)R^{\Lambda}
(\beta) $, we have
\begin{align}
F(z)x_{n+1} = F(z(x_n \otimes 1)) + E(z).
\end{align}
\end{lemma}

\begin{proof}Let $z = a \otimes b\in R(\beta)e(\beta-\alpha_i,i) \otimes_{R(\beta-\alpha_i)}
e(\beta-\alpha_i,i)R^{\Lambda}(\beta) $,
where $a \in R(\beta)e(\beta-\alpha_i,i)$ and $b \in e(\beta-\alpha_i,i)R^{\Lambda}$. Then
$$ F(z) = a \tau_n b, \qquad E(z)=a \P_i(x_n,x_{n+1}) b.$$
Thus
\begin{align*}
F(z)x_{n+1} & = a \tau_n b x_{n+1} =  a \tau_n  x_{n+1} b  = a(x_n \tau_n + \P_i(x_n,x_{n+1}))b \\
& = a x_n \tau_n b + a \P_i(x_n,x_{n+1})b \\
& = F(ax_n \otimes b) + E(z) =F(z(x_n \otimes 1))+E(z).
\end{align*}
\end{proof}

By Proposition \ref{Prop: twist by tau n}, we have
\begin{equation} \label{Eq: Decom wrt F}
\begin{aligned}
& e(\beta,i)R(\beta+\alpha_i)e(\beta,i) \otimes_{R(\beta)} R^{\Lambda}(\beta) \\
& \ = F\bl(R(\beta)e(\beta-\alpha_i,i)\otimes_{R(\beta-\alpha_i)}e(\beta-\alpha_i,i)R^{\Lambda}(\beta)\br) \oplus
      (R^{\Lambda}(\beta) \otimes \k[t_i])e(\beta,i),
\end{aligned}
\end{equation}
where $t_i=x_{n+1}$. Using the decomposition $\eqref{Eq: Decom wrt F}$, we write
\begin{align} \label{Eq: Decom Im P wrt F }
P(e(\beta,i)\tau_n \cdots \tau_1 x^k_1 e(i,\beta)) = F(\psi_k) + \varphi_k
\end{align}
for uniquely determined $\psi_k \in K'_0$ and $\varphi_k \in \k[t_i] \otimes R^{\Lambda}(\beta)$.

Using $\eqref{Eq: varphi k}$, we have
$$ A(t^k_i)= AB(e(\beta,i)\tau_n\cdots\tau_1 x^k_{1} e(i,\beta))
           = CP(e(\beta,i)\tau_n\cdots\tau_1 x^k_{1} e(i,\beta)) = \varphi_k.$$
Thus one can verify that the definition of $\varphi_k$ coincides with the definition
given in $\eqref{Eq: varphi k}$.

Since
\begin{align*}
F(\psi_{k+1})+\varphi_{k+1}
&= P \bl(e(\beta,i)\tau_n \cdots \tau_1 x_1^{k+1} e(i,\beta)\br)
= P\bl(e(\beta,i)\tau_n \cdots \tau_1 x_1^{k}  e(i,\beta)\br)x_{n+1} \\
&= \bl(F(\psi_{k})+\varphi_{k}\br)x_{n+1}
 = F(\psi_k(x_n \otimes 1))+ E(\psi_k)+\varphi_k t_i,
\end{align*}
we have
\begin{align} \label{Eq: formulas psi varphi}
\psi_{k+1}=\psi_k(x_n \otimes 1), \qquad \varphi_{k+1} =E(\psi_k)+\varphi_k t_i.
\end{align}

Now we will prove Proposition \ref{Prop: varphi k}. By Lemma \ref{Lem : B nu}, we have
$$g_n\cdots g_1 x^k_1e(i,\nu)\tau_1 \cdots \tau_n  =
x^{k}_{n+1}a^{\Lambda}_i(x_{n+1})
\prod_{\substack{1 \le a \le n, \\ \nu_a \neq i}}\Q_{i,\nu_a}(x_{n+1},x_{a})
\prod_{\substack{1 \le a \le n, \\ \nu_a = i}}\P_i(x_{n+1},x_{a})\P_i(x_{a},x_{n+1})e(\nu,i)$$
in $e(\beta,i)R(\beta+\alpha_i)e(\beta,i) \otimes R^{\Lambda}(\beta)$, which implies
\begin{align*}
AB(g_n\cdots g_1 x^k_1 e(i,\nu)) &= C\Bigl( x^{k}_{n+1}a^{\Lambda}_i(x_{n+1})
\prod_{\substack{1 \le a \le n, \\ \nu_a \neq i}}\Q_{i,\nu_a}(x_{n+1},x_{a})
\prod_{\substack{1 \le a \le n, \\ \nu_a = i}}\P_i(x_{n+1},x_{a})\P_i(x_{a},x_{n+1}) \Bigr)
e(\nu,i)\\
&= t_i^{k}a^{\Lambda}_i(t_i)
\prod_{\substack{1 \le a \le n, \\ \nu_a \neq i}}\Q_{i,\nu_a}(t_i,x_{a})
\prod_{\substack{1 \le a \le n, \\ \nu_a = i}}\P_i(t_i,x_{a})\P_i(x_{a},t_i)e(\nu).
\end{align*}
On the other hand, since $B$ is the map  taking the coefficient of
$\tau_n \cdots \tau_1$, we have
\begin{align*}
B(g_n\cdots g_1 x^k_1  e(i,\nu))
&= B\left( \prod_{\nu_a =i}(-(x_{n+1}-x_a)^2)x^k_{n+1} e(\nu,i) \tau_n \cdots \tau_1 \right) \\
&= t^k_i  \prod_{\nu_a =i} (-(t_{i}-x_a)^2 )e(\nu).
\end{align*}
Thus we have
\begin{align} \label{Eq: the image A}
A(t^k_i \prod_{\nu_a =i}(t_i-x_a)^2e(\nu) = (-1)^p  t_i^{k} a^{\Lambda}_i(t_i)
\prod_{\substack{1 \le a \le n, \\ \nu_a \neq i}}\Q_{i,\nu_a}(t_i,x_{a})
\prod_{\substack{1 \le a \le n, \\ \nu_a = i}}\P_i(t_i,x_{a})\P_i(x_{a},t_i)e(\nu).
\end{align}

Set
\begin{align*}
& \mathsf{S} = \sum_{\nu \in I^{\beta}}
\prod_{\nu_a =i}(t_i-x_a)^2 e(\nu) \in \k[t_i] \otimes R^{\Lambda}(\beta) , \\
& \mathsf{F} = \gamma (-1)^p  a^{\Lambda}_i(t_i) \sum_{\nu \in I^{\beta}} \left(
\prod_{\substack{1 \le a \le n, \\ \nu_a \neq i}}\Q_{i,\nu_a}(t_i,x_{a})
\prod_{\substack{1 \le a \le n, \\ \nu_a = i}}\P_i(t_i,x_{a})\P_i(x_{a},t_i) e(\nu) \right)
\in \k[t_i] \otimes R^{\Lambda}(\beta).
\end{align*}
Then they are monic polynomials in $t_i$ of degree $2p$ and $\langle h_i, \lambda \rangle+2p$, respectively ($p \seteq \Supp_i(\beta)$).
Note that they are contained in the center of $\k[t_i] \otimes R^{\Lambda}(\beta)$. Then $\eqref{Eq: the image A}$
can be expressed as the following form:
\begin{align} \label{Eq: Rel F and S}
 \gamma A(t_i^{k}\mathsf{S})=t_i^{k}\mathsf{F}.
\end{align}

Note that
\begin{itemize}
\item if $\Supp_i(\beta)=0$, then $K'_0=0$ and
\item if $i \in \Iim$ such that $\langle h_i, \lambda \rangle=0$
and $\Supp_i(\beta)>0$, then $R^{\Lambda}(\beta)=0$
(see Lemma~\ref{Lem: imaginary integrable}).
\end{itemize}
Thus, to prove Proposition~\ref{Prop: varphi k}, we may assume that
\begin{align} \label{Eq: assumption}
\parbox{40ex}{
$\Supp_i(\beta)>0$ and \\[1ex]
if $i \in \Iim$, then $\langle h_i, \lambda \rangle>0$.}
\end{align}

\begin{lemma} \label{Lem: Rel between F and S}
For any $k \ge 0$, we have
\begin{align} \label{Eq: remaining term}
 t^k_i \mathsf{F}= (\gamma \varphi_k)\mathsf{S}+\mathsf{h}_k,
\end{align}
where $\mathsf{h}_k\in R^\Lambda(\beta)[t_i]$
is a polynomial in $t_i$  of degree $<2p$.
In particular, $\gamma \varphi_k$ coincides with the
quotient of $t^k_i \mathsf{F}$ by $\mathsf{S}$.
\end{lemma}

\begin{proof}
By $\eqref{Eq: formulas psi varphi}$,
$A(t_i^{k+1})-A(t_i^{k})t_i \in R^{\Lambda}(\beta)$, which implies
\begin{align} \label{Eq: Degree 1 }
A(a t_i)-A(a)t_i \in R^{\Lambda}(\beta)[t_i] \text{ is of degree $\le 0$ in $t_i$
for any $a \in R^{\Lambda}(\beta)[t_i]$.}
\end{align}
We will show
\begin{equation}  \label{Eq: Degree 2 }
\begin{aligned}
& \parbox{70ex}{for any polynomial $f\in R^{\Lambda}(\beta)[t_i]$ in
$t_i$ of degree $m$ and
    $a \in R^{\Lambda}(\beta)[t_i]$,} \\
&\text{$A(a f)- A(a)f$ is of degree $<m$.}
\end{aligned}
\end{equation}
We will use induction on $m$. By the fact that $A$ is $R^{\Lambda}(\beta)$-linear and $\eqref{Eq: Degree 1 }$,
 it holds for $m=0$ and $1$.
Thus it suffices to show \eqref{Eq: Degree 2 } when $f=t_i g$ and \eqref{Eq: Degree 2 } is true for $g$.
Then
\begin{align*}
A(af)-A(a)f = (A(a t_i g) - A(a t_i)g) + ( A(a t_i) - A(a)t_i) g.
\end{align*}
Then the first term
is of degree $<\deg(g)$ in $t_i$ and
the second term is of degree $<\deg(g)+1$.
Hence we prove $\eqref{Eq: Degree 2 }$.
Thus we have
$$ t^k_i \gamma^{-1}\mathsf{F} - \varphi_k \mathsf{S} = t^k_i \gamma^{-1}\mathsf{F} -A(t^k_i)\mathsf{S}
= A(t^k_i\mathsf{S})-A(t^k_i)\mathsf{S}$$
by $\eqref{Eq: Rel F and S}$ and it is of degree $<2p$
by applying $\eqref{Eq: Degree 2 }$
for $f=\mathsf{S}$.
\end{proof}

Thus by Lemma \ref{Lem: Rel between F and S}, we can conclude that
$\gamma \varphi_k$ is a monic polynomial in $t_i$ of degree $\langle
h_i, \lambda \rangle +k$, which completes the proof of Proposition
\ref{Prop: varphi k}.

\bigskip
\noi \textbf{Proof of Theorem \ref{Thm: Main}:} By the Snake Lemma,
we have the following exact sequence
$$0 \to {\rm Ker}A \to q^{-a_{ii}}_iF^{\Lambda}_iE^{\Lambda}_iR^{\Lambda}(\beta)
\to E^{\Lambda}_iF^{\Lambda}_iR^{\Lambda}(\beta) \to {\rm Coker}A \to 0.$$

If $\langle h_i,\lambda \rangle \ge 0$, by Proposition \ref{Prop: varphi k}, we have
$$ {\rm Ker}A=0, \qquad \bigoplus^{\langle h_i,\lambda \rangle-1}_{k=0} \k t^k_i \otimes R^{\Lambda}(\beta)
\overset{\sim}{\to} {\rm Coker}A.$$
Hence we obtain
$$q_i^{-a_{ii}}F^{\Lambda}_iE^{\Lambda}_i \oplus
\bigoplus^{\langle h_i,\lambda \rangle-1}_{k=0} q_i^{2k}{\rm Id} \overset{\sim}{\to}
E^{\Lambda}_iF^{\Lambda}_i,$$
which is the proof the statement of Theorem \ref{Thm: Main} (1).

If $\langle h_i,\lambda \rangle <0$, then $i \in \Ire$. In this
case, the proof is the same as in \cite[Theorem 5.2 (b)]{KK11}. \qed
\vskip 1em

We define the modified functors $\mathcal{E}^{\Lambda}_i$ and $\mathcal{F}^{\Lambda}_i$ on $\Mod(R)$:
$$\mathcal{E}^{\Lambda}_i =E^{\Lambda}_i, \quad
\mathcal{F}^{\Lambda}_i= q_i^{1-\langle h_i,\Lambda-\beta\rangle}
F^{\Lambda}_i.$$ Then by applying degree shift functor
$q_i^{1-\langle h_i,\Lambda-\beta \rangle}$ to the equations
$\eqref{Thm: Comm E_i Lam F_j Lam}$, $\eqref{Eq: Comm E_i Lam F_i
Lam 1}$ and $\eqref{Eq: Comm E_i Lam F_i Lam 2}$, we obtain the
natural isomorphisms
\begin{equation} \label{Eq: The com rel}
\begin{aligned}
& \mathcal{E}^{\Lambda}_i \mathcal{F}^{\Lambda}_j \simeq \mathcal{F}^{\Lambda}_j \mathcal{E}^{\Lambda}_i \quad \text{ if } i \neq j, \\
& \mathcal{E}^{\Lambda}_i \mathcal{F}^{\Lambda}_i \simeq \mathcal{F}^{\Lambda}_i \mathcal{E}^{\Lambda}_i \oplus
\dfrac{q_i^{\langle h_i,\Lambda-\beta \rangle}-q_i^{-\langle h_i,\Lambda-\beta \rangle}}{q_i-q^{-1}_i}
{\rm Id}\quad\text{if $\langle h_i,\Lambda-\beta \rangle\ge0$,}\\
&\mathcal{E}^{\Lambda}_i \mathcal{F}^{\Lambda}_i\oplus
\dfrac{q_i^{-\langle h_i,\Lambda-\beta \rangle}-q_i^{\langle
h_i,\Lambda-\beta \rangle}}{q_i-q^{-1}_i}{\rm Id} \simeq
\mathcal{F}^{\Lambda}_i \mathcal{E}^{\Lambda}_i \quad\text{if
$\langle h_i,\Lambda-\beta \rangle\le0$}
\end{aligned}
\end{equation}
on $\Mod(R^{\Lambda}(\beta))$.
 Now, assume that $\k_0$ is a field.
Then,
as operators on
$[\Proj(R^{\Lambda})]$ and $[\Rep(R^{\Lambda})]$, they satisfy the
commutation relations
$$[\mathcal{E}^{\Lambda}_i, \mathcal{F}^{\Lambda}_j] = \delta_{i,j}\dfrac{K_i-K^{-1}_i}{q_i-q^{-1}_i},$$
where
$$K_i|_{[\Proj(R^{\Lambda}(\beta))]} \seteq q_i^{\langle h_i,\Lambda-\beta \rangle}, \quad
K_i|_{[\Rep(R^{\Lambda}(\beta))]}\seteq q_i^{\langle h_i,\Lambda-\beta \rangle}.$$

Combining  Lemma \ref{Lem:
nilpotency}, Lemma \ref{Lem: imaginary integrable} and Theorem
\ref{Thm: main result of KOP11} as in \cite[Section 6]{KK11},
we obtain a categorification of the
irreducible highest weight $U_q(\g)$-module $V(\Lambda)$:

\begin{theorem} \label{Thm: Categoeification}
 If $a_{ii} \neq 0$ for all $i \in I$, then there exist $U_\A(\g)$-module isomorphisms
 $$[\Proj(R^{\Lambda})] \simeq V_{\A}(\Lambda) \quad \text{ and } \quad
 [\Rep(R^{\Lambda})] \simeq V_{\A}(\Lambda)^{\vee}. $$
\end{theorem}
\vskip 2em
%%%%%%%%%%%%%%%%%%%%%%%%%%%%%%%%%%%%%%%%%%%%%%%%%%%%%%%%%%%%%%%%%%%

\bibliographystyle{amsplain}

%%%%%%%%%%%%%%%%%%%%%%%%%%%%%%%%%%%%%%%%%%%%%%%%%%%%%%%%%%%%%%%%%%%

\end{document}